\documentclass[12pt]{amsart}
\usepackage{fullpage}

\usepackage{amsfonts}
\usepackage{amsthm}
\usepackage{enumerate}
\usepackage{verbatim}
\usepackage{graphicx}

\newcommand{\R}{\mathbb{R}}
\newcommand{\C}{\mathbb{C}}
\newcommand{\N}{\mathbb{N}}

\newcommand{\Z}{\mathbb{Z}}

\newcommand{\ve}{\varepsilon}
\newcommand{\lan}{\langle}
\newcommand{\ran}{\rangle}
\DeclareMathOperator{\spa}{span}
\DeclareMathOperator{\tr}{tr}
\DeclareMathOperator{\diag}{diag}
\DeclareMathOperator{\maxroot}{maxroot}

\newcommand{\E}[2]{\mathbb{E}_{#1}\left[ #2 \right] }
\renewcommand{\P}[2]{\mathbb{P}_{#1}\left( #2 \right)}

\renewcommand{\Im}{\operatorname{Im}}
\newcommand{\bI}{\mathbf I}
\newcommand{\ch}{\mathbf 1}

\numberwithin{equation}{section}
\theoremstyle{plain} 

\newtheorem{theorem}{Theorem}[section]
\newtheorem*{theorem*}{Theorem}
\newtheorem{corollary}[theorem]{Corollary}
\newtheorem{lemma}[theorem]{Lemma}

\newtheorem*{problem}{Problem}

\theoremstyle{definition}
\newtheorem{definition}{Definition}[section]

\theoremstyle{remark}
\newtheorem{remark}{Remark}[section]

\begin{document}

\title{The Kadison-Singer Problem}

\author{Marcin Bownik}

\address{Department of Mathematics, University of Oregon, Eugene, OR 97403--1222, USA}
\curraddr{Institute of Mathematics, Polish Academy of Sciences, ul. Wita Stwosza 57,
80--952 Gda\'nsk, Poland}
\email{mbownik@uoregon.edu}

\date{\today}



\thanks{The author is grateful for useful conversations on the Kadison-Singer problem with Bernhard Bodmann, Jean Bourgain, Pete Casazza, Bill Johnson, Adam Marcus, Gideon Schechtman, Darrin Speegle, and participants of the AIM workshop ``Beyond Kadison-Singer: paving and consequences'' in December 2014 and the MSRI workshop ``Hot Topics: Kadison-Singer, Interlacing Polynomials, and Beyond'' in March 2015.
The author was partially supported by NSF grant DMS-1265711 and by a grant from the Simons Foundation \#426295. }

\begin{abstract} We give a self-contained presentation of results related to the Kadison-Singer problem, which was recently solved by Marcus, Spielman, and Srivastava \cite{MSS}. This problem connects with an unusually large number of research areas including: operator algebras (pure states), set theory (ultrafilters), operator theory (paving), random matrix theory, linear and multilinear algebra, algebraic combinatorics (real stable polynomials), algebraic curves, frame theory, harmonic analysis (Fourier frames), and functional analysis. 

\end{abstract}

\maketitle

\section{Introduction}

The goal of this paper is to give a self-contained presentation of mathematics involved in Kadison-Singer problem \cite{KS}. This problem was shown to be equivalent to a large number of problems such as: Anderson paving conjecture \cite{An1, An2, An3}, Bourgain--Tzafriri restricted invertibility conjecture \cite{BT1, BT2, BT3}, Akemann--Anderson projection paving conjecture \cite{AA}, Feichtinger conjecture \cite{BS, CCLV, Gr}, $R_\epsilon$ conjecture \cite{CT}, and Weaver conjecture \cite{We}. The breakthrough solution of the Weaver conjecture \cite{We} by Marcus, Spielman, and Srivastava \cite{MSS} has validated all of these conjectures. 

While a lot has been written about Kadison-Singer problem before it was solved \cite{BT3, BS, CCLV, CE, CEKP, CFTW, CT} and after its solution \cite{CT2, Mat, Tan, Tao, Ti, Va}, we believe that there is still a space and an interest for yet another presentation which would present  consequences of the solution of Kadison-Singer problem in an optimized form. We have aimed at presenting the material in a modular form as a sequence of implications which are largely independent of each other. This paper is the result of the author's attempts at achieving this goal. It is based on a series of lectures on the subject given by the author at University of Oregon, University of Gda\'nsk, Institute of Mathematics of Polish Academy of Sciences, and Tel Aviv University. The author is grateful for the interest and hospitality of these institutions.

The general outline of the paper can be described by the following diagram:
\[
\begin{matrix}

(KS) \Leftarrow (PB) \Leftarrow (PS) \Leftarrow (PR) \Leftarrow (PP_{\frac 12}) \Leftarrow   & \hskip -8pt (KS^\infty_r)  & \hskip -13pt  \Leftarrow   & \hskip -10pt (KS_r)   & \hskip-10pt  \Leftarrow \hskip -2pt   (MSS) \hskip -2pt \Leftarrow \hskip -2pt  (MCP) \\
&\Downarrow& &\hskip -15pt \Downarrow&\\
 \hfill (PP_\delta) \Leftarrow & \hskip -5pt (KS_2^\infty) &  & \hskip -10pt (FEI)  & \hskip -5pt   \Rightarrow  (R_\epsilon)  \Rightarrow (BT) \hfill
\end{matrix}
\]
The above symbols represent abbreviations used throughout the paper. The most important are: the original Kadison-Singer problem $(KS)$, Weaver's conjecture $(KS_r)$, Marcus-Spielman-Srivastava solution $(MSS)$, and its mixed characteristic polynomial formulation $(MCP)$. In this paper we will prove all of the above implications including the proof of the core statement $(MCP)$.

\subsection{Notes} The existing literature discussing the solution of the Kadison-Singer problem is quite varied. It gives a deeper appreciation for the many areas of mathematics this problem has touched. Tao \cite{Tao} has written a beautiful exposition containing a simplified proof of the solution.  Tanbay \cite{Tan} has given a nice entertaining historical perspective on the Kadison-Singer problem.  Matheron \cite{Mat} gives a long and exhaustive exposition (in French), primarily from the viewpoint of operator algebras. Valette \cite{Va} has written a Bourbaki exposition (in French). Casazza, who has worked and popularized the Kadison-Singer problem, has written a joint paper with Tremain \cite{CT2} discussing consequences of the solution. Timotin \cite{Ti} gives another presentation of the proof of the Kadison-Singer problem. Harvey \cite{Ha} gives a gentle introduction aimed at readers without much background in functional analysis or operator theory. Finally, the book by Stevens \cite{St} contains a relatively elementary and self-contained account of the Kadison-Singer problem and its proof.

\section{From Kadison-Singer Problem to Weaver's conjecture}

\subsection{Kadison-Singer Problem} We start with the Kadison-Singer problem \cite{KS}, now a theorem, which was originally formulated in 1959.

\begin{definition}
Let $\mathcal D \subset \mathcal B(\ell^2(\N))$ be the algebra of diagonal matrices. A state $s: \mathcal D \to \C$ is a positive  bounded linear functional $(A\ge 0 \implies s(A) \ge 0)$ such that $s(\mathbf I)=1$. A state is pure if it is not a convex combination of other states.
\end{definition}

\begin{theorem*}[$KS$] \label{KS}
Let $s: \mathcal D \to \C$ be a pure state. Then, there exists a unique state $\tilde s: \mathcal B(\ell^2(\N)) \to \C$ that extends $s$.
\end{theorem*}

The original formulation of Kadison-Singer problem involves the concept of a maximal abelian self-adjoint (MASA) subalgebra $\mathcal A$ of $\mathcal B(\mathcal H)$, where $\mathcal H$ is an infinite-dimensional separable Hilbert space. Kadison and Singer \cite{KS} have shown that every MASA $\mathcal A$ decomposes into discrete and continuous parts. More precisely, there exists an orthogonal decomposition $\mathcal H = \mathcal H_d \oplus \mathcal H_c$ with $P$ denoting the orthogonal projection of $\mathcal H$ onto $\mathcal H_d$ such that:
\begin{itemize} 
\item $\mathcal A_d = \{ P A|_{\mathcal H_d}: A\in \mathcal A\}$ is a discrete MASA, i.e., $\mathcal A_d$ is the commutant of the set of its minimal projections,
\item $\mathcal A_c = \{ (\mathbf I -P) A|_{\mathcal H_c}: A\in \mathcal A\}$ is a continuous MASA, i.e., $\mathcal A_c$ contains no minimal projections.
\end{itemize}
An example of a continuous MASA $\mathcal A_c$ are multiplication operators on $L^2[0,1]$ by functions in $L^\infty[0,1]$. Kadison and Singer have shown that as long as $\mathcal H_c \ne  \{0\}$, there exists a pure state on $\mathcal A$, which has non-unique state extensions on $\mathcal B(\mathcal H)$. They have hinted \cite[\S5]{KS} that the same might hold in general though they were careful to state this in a form of a question, rather than a conjecture.

\begin{problem}[Kadison-Singer] \label{KSO}
Let $\mathcal H$ be an infinite-dimensional separable Hilbert space and let $\mathcal A$ be a discrete maximal abelian self-adjoint subalgebra (MASA) of $\mathcal B(\mathcal H)$. Does every pure state on $\mathcal A$ extend to a unique pure state on $\mathcal B( \mathcal H)$?
\end{problem}

One can show that every discrete MASA $\mathcal A$ is unitarily equivalent with the diagonal matrix algebra $\mathcal D$ in Theorem $(KS)$.  That is, there exists a unitary $U: \mathcal H \to \ell^2(\N)$ such that $\mathcal A = U^*\mathcal D U$. Hence, Theorem $(KS)$ gives an affirmative answer to the Kadison-Singer problem.

Clearly, the diagonal matrix algebra $\mathcal D$ is isometrically isomorphic with $\ell^\infty(\N)$. That is, $x\in \ell^\infty(\N)$ corresponds to diagonal operator $\diag(x)$ with sequence $x$ on the main diagonal. Pure states on $\ell^\infty(\N)$ can be described in terms of the ultrafilters. 

\begin{definition} Let $\mathcal F$ be a collection of non-empty subsets of $\N$. We say that $\mathcal F$ is a filter if:
\begin{enumerate}
\item if $F_1,\ldots, F_n\in \mathcal F$, $n\ge 1$, then $F_1 \cap \ldots \cap F_n \in \mathcal F$,
\item if $F \in \mathcal F$ and $F\subset G \subset \N$, then $G \in \mathcal F$.
\end{enumerate}
We say that $\mathcal U$ is an ultrafilter, if it is a maximal filter with respect to the inclusion partial order. Equivalently, 
\begin{enumerate}
\item[(iii)] for any $A\subset \N$, either $A\in \mathcal U$ or $\N \setminus A \in \mathcal U$.
\end{enumerate}
\end{definition}

Given an ultrafliter $\mathcal U$, we can define the concept of a limit of a bounded sequence on $\N$.

\begin{definition} Fix an ultrafilter $\mathcal U$. Let $x=(x_j)_{j\in\N} \in \ell^\infty(\N)$. For any subset $A\subset \N$, define
\[
C_A= \overline{ \{x_j : j \in A \}} \subset \C.
\]
We define 
\[
\lim_{\mathcal U} x =x_0 \iff \{x_0\} = \bigcap_{A \in\mathcal U} C_A.
\]
\end{definition}

It is easy to see that the above limit is always well-defined. The intersection of any finite family of compact sets $C_A$, $A\in\mathcal U$ is non-empty. Hence, the entire intersection is non-empty as well. Moreover, it consists of exactly one point. On the contrary, if it contained two points $x_0 \ne \tilde x_0$, then we would consider the set
\[
A= \{ j\in \N: |x_j-x_0| < |x_0 - \tilde x_0|/2 \}.
\]
Then by the ultrafilter property (iii), we have two possibilities. Either $A \in \mathcal U$, which forces $\tilde x_0$ outside the intersection $\bigcap_{A \in\mathcal U} C_A$, or $\N \setminus A \in \mathcal U$, which forces $\tilde x_0$ outside. Either way, the above intersection must be a singleton.

\begin{lemma} There is one-to-one correspondence between pure state on the algebra $\mathcal D$ of diagonal matrices and ultrafilters on $\N$. More precisely, each pure state $s$ on $\mathcal D  \cong \ell^\infty(\N)$ is of the form
\begin{equation}\label{lim}
s(\diag(x)) = \lim_\mathcal U x \qquad\text{for all }x\in \ell^\infty(\N)
\end{equation}
for some unique ultrafilter $\mathcal U$.
\end{lemma}

\begin{proof} Suppose that $s$ is a pure state on $\ell^\infty(\N)$. For $A\subset \N$, let $P_A$ be the orthogonal projection of $\ell^2(\N)$ onto $\overline\spa\{e_j: j \in A\}$, where $\{e_j\}_{j\in \N}$ is a standard o.n. basis of $\ell^2(\N)$. Define
\[
\mathcal U = \{ A \subset \N: s(P_A) =1\}.
\]
We claim that $\mathcal U$ is an ultrafilter. This can be shown in two ways. The \v Cech-Stone compactification of $\N$ is a maximal compact Hausdorff space $\beta \N$, which contains $\N$ as a dense subset. By the universal property of $\beta \N$, the space $\ell^\infty(\N)$ is isometrically isomorphic with $C(\beta \N)$. By the Riesz representation theorem, positive functionals on $C(\beta \N)$ are identified with positive regular Borel measures. Hence, $s$ corresponds to a probabilistic measure on $\beta \N$. In addition, since $s$ is pure, this measure must be a point measure on $\beta \N$. In particular, $s$ is commutative. Hence, $s(P_A)=s(P_A)^2$, which implies that $s(P_A) \in \{0,1\}$. Likewise,
\begin{equation}\label{AB}
s(P_{A\cap B}) = s(P_A)s(P_B) \qquad\text{for any }A,B \subset \N,
\end{equation}
which implies that $\mathcal U$ is an ultrafilter.

This can also be seen by a direct argument as follows. Since $P_A$ and $\mathbf I -P_A=P_{\N \setminus A}$ are both positive, we have $0\le s(P_A) \le s(\mathbf I)=1$. Suppose that $s(P_A)=\theta$ for some $0<\theta<1$. Then, we can write $s=\theta s_1+ (1-\theta) s_2$, where $s_1(T)=\frac1\theta s(P_A T)$ and $s_2(T)=\frac{1}{1-\theta}s((\mathbf I - P_A)T)$ for $T\in \mathcal D$. It is easy to show that $s_1$ and $s_2$ are states on $\mathcal D$, which contradicts that $s$ is a pure state. Consequently, $s(P_A) \in \{0,1\}$. By the positivity of $s$, it is clear that \eqref{AB} holds if either $s(P_A)=0$ or $s(P_B)=1$. Now, if $s(P_A)=s(P_B)=1$, then 
\[
s(P_{\N \setminus (A\cap B)}) = s(P_{(\N \setminus A) \cup (\N \setminus B)}) \le s(P_{\N \setminus A})+s(P_{\N \setminus B}) = 0,
\]
This shows \eqref{AB}, which again implies that $\mathcal U$ is an ultrafilter.

Every $x\in \ell^\infty(\N)$ can be approximated in norm by simple functions, i.e., finite linear combinations of indicator functions $\ch_{A_i}$ for disjoint subsets $A_i \subset \N$, $i=1,\ldots, n$. By definition
\[
s(\diag(\ch_{A_i}))= s(P_{A_i}) = \lim_{\mathcal U} \ch_{A_i}= \begin{cases} 1 & A_i \in\mathcal U,\\
0 & \text{otherwise.}
\end{cases}
\]
Thus, \eqref{lim} holds for indicator functions and by the density argument for all $x\in \ell^\infty(\N)$. This implies that two distinct pure states must correspond to distinct ultrafilters, which shows one-to-one correspondence.
\end{proof}

\subsection{Paving conjectures} To crack Theorem ($KS$) Anderson has proposed the concept of paving. We will adopt the following definition.

\begin{definition} Let $T \in \mathcal B(\ell^2(I))$, where $I$ is at most countable. We say that $T$ has $(r,\ve)$-paving if there exists a partition $\{A_1, \ldots ,  A_r\}$ of $I$ such that
\begin{equation}\label{pd}
||P_{A_j} T P_{A_j} || \le \ve ||T|| \qquad\text{for }j=1,\ldots,r.
\end{equation}
Here, for $A\subset I$, let $P_A$ be the orthogonal projection of $\ell^2(I)$ onto $\overline\spa\{e_i: i \in A\}$, where $\{e_i\}_{i\in I}$ is a standard o.n. basis of $\ell^2(I)$.
\end{definition}

The following result states the paving conjecture for bounded operators with zero diagonal.

\begin{theorem*}[$PB$] For every $\ve>0$, there exists $r=r(\ve)$ such that every $T \in \mathcal B(\ell^2(I))$ with zero diagonal can be $(r,\ve)$-paved.
\end{theorem*}

We are now ready to establish the first implication in our scheme.

\begin{lemma} $(PB) \implies (KS)$.
\end{lemma}

\begin{proof} Let $\mathbb E : \mathcal B(\ell^2(\N)) \to \mathcal D$ be the non-commutative conditional expectation which erases all off-diagonal entries. That is, for any $T \in \mathcal B(\ell^2(\N))$, let $\mathbb E(T)$ be the diagonal operator which has the same diagonal entries as $T$. Let $s$ be any pure state on $\mathcal D$. It is easy to show that $\tilde s(T) = s(\mathbb E(T))$, $T \in \mathcal B(\ell^2(\N))$, defines a state extending $s$. Hence, the difficult part is showing the uniqueness. 

Let $\tilde s: \mathcal B(\ell^2(\N)) \to \C$ be any state extending $s$. Since
\[
\tilde s(T) = \tilde s(T - \mathbb E (T)) + s(\mathbb E (T)),
\]
it suffices to show that 
\begin{equation}\label{uq}
\tilde s(T)=0 \qquad\text{for all } T \in \mathcal B(\ell^2(\N)) \text{ with } \mathbb E(T)= \mathbf 0.
\end{equation}
By $(PB)$ for any $\ve>0$ we can find $A_1, \ldots, A_r$ such that \eqref{pd} holds. By the ultrafilter property, there exists $ j_0 \in [r]:=\{1,\ldots,r\} $ such that
\[
s(P_{A_j}) = \delta_{j,j_0} \qquad \text{for }j \in [r].
\]
One can easily verify that
\[
\lan T_1, T_2 \ran := \tilde s( T_1 T_2^*) \qquad T_1,T_2 \in \mathcal B(\ell^2(\N)),
\]
defines a semidefinite inner product on $\mathcal B(\ell^2(\N))$. In particular, by the Cauchy-Schwarz inequality we have
\[
|\tilde s( T_1 T_2^*)|^2 \le \tilde s( T_1 T_1^*) \tilde s( T_2 T_2^*).
\]
Thus, for any $j\ne j_0$ and $R\in \mathcal B(\ell^2(\N))$ we have
\[
0=\tilde s(RP_{A_j}) = \tilde s(P_{A_j}R).
\]
We conclude that
\[
\tilde s(T) = \sum_{j=1}^r \tilde s(T P_{A_{j}}) = \tilde s(T P_{A_{j_0}}) = \tilde s(P_{A_{j_0}}T P_{A_{j_0}}).
\]
Thus, $|\tilde s(T)| \le ||P_{A_{j_0}}T P_{A_{j_0}}|| \le \ve ||T||$. Since $\ve>0$ is arbitrary, this shows \eqref{uq}.
\end{proof}

Paving conjectures can be formulated for smaller classes  of operators than bounded $(PB)$ such as: self-adjoint operators $(PS)$, reflections $(PR)$, and projections $(PP_{\frac12})$.

\begin{theorem*}[$PS$] For every $\ve>0$, there exists $r=r(\ve)$ such that every self-adjoint operator $S$ on $\ell^2(I)$ with zero diagonal can be $(r,\ve)$-paved.
\end{theorem*}

\begin{theorem*}[$PR$] For every $\ve>0$, there exists $r=r(\ve)$ such that every reflection $R$ on $\ell^2(I)$,  i.e., $R=R^*$ and $R^2=\mathbf I$, with zero diagonal can be $(r,\ve)$-paved.
\end{theorem*}

\begin{theorem*}[$PP_{\frac12}$] For every $\ve>0$, there exists $r=r(\ve)$ such that every projection $P$ on $\ell^2(I)$,  i.e., $P=P^*$ and $P^2=P$, with all diagonal entries equal to $ \frac{1}{2}$ can be $(r,\frac{1+\ve}2)$-paved.
\end{theorem*}

While the implication $(PS) \implies (PB)$ is trivial, we need to show the converse implication. At the same time we shall keep track how the paving parameters $(r,\ve)$ are affected by these implications.

\begin{lemma} $(PS)$ holds for $(r,\ve)$ $\implies$ $(PB)$ holds for $(r^2,2\ve)$.
\end{lemma}

\begin{proof}
Take any $T \in \mathcal B(\ell^2(I))$ with $\mathbb E(T)=\mathbf 0$. We decompose it as sum of self-adjoint and skew-adjoint operators
\[
T= S_1+S_2 \qquad \text{where } S_1=\frac{T+T^*}2, \ S_2=\frac{T-T^*}2.
\]
By the paving property $(PB)$ for $S_1$ and $iS_2=(iS_2)^*$ we can find partitions $\{A_1,\ldots, A_r\}$ and $\{B_1,\ldots,B_r\}$ such that
\[
||P_{A_i \cap B_j} S_1 P_{A_i \cap B_j} || \le ||P_{A_i} S_1 P_{A_i} || \le \ve ||S_1|| \le \ve ||T|| \qquad i,j \in [r].
\]
Since the same estimate holds for $S_2$ we have
\[
||P_{A_i \cap B_j} T P_{A_i \cap B_j} || \le ||P_{A_i \cap B_j} S_1 P_{A_i \cap B_j} || +||P_{A_i \cap B_j} S_2 P_{A_i \cap B_j} ||   \le 2 \ve ||T|| \quad i,j \in [r].
\]
Hence, the partition $\{A_i \cap B_j\}_{i,j \in [r]}$ yields $(r^2,2\ve)$ paving of $T$.
\end{proof}

\begin{lemma} $(PR)$ holds for $(r,\ve)$ $\implies$ $(PS)$ holds for $(r,\ve)$.
\end{lemma}

\begin{proof}
Take any $S=S^* \in \mathcal B(\ell^2(I))$ with $\mathbb E(S)=\mathbf 0$. Without loss of generality assume that $||S||=1$. Consider an operator $R$ on $\ell^2(I) \oplus \ell^2(I)$ given by
\[
\begin{bmatrix} S & \sqrt{\mathbf I - S^2} \\ \sqrt{\mathbf I - S^2} & -S \end{bmatrix}.
\]
A direct calculation shows that $R^2=\mathbf I$, $R=R^*$, and $\mathbb E(R)=\mathbf 0$. That is, $R$ is a reflection on $\ell^2(I \cup I')$, where $I'$ is a copy of the index set $I$. By $(PR)$ there exists a partition of $I\cup I'$ which yields $(r,\ve)$ paving of $R$. Restricting this partition to $I$ yields $(r,\ve)$ paving of $S$.
\end{proof}

\begin{lemma} $(PP_{\frac12})$ holds for $(r,\frac{1+\ve}2)$ $\implies$ $(PR)$ holds for $(r^2,\ve)$.
\end{lemma}

\begin{proof}
Take any reflection $R\in \mathcal B(\ell^2(I))$ with $\mathbb E(R)=\mathbf 0$. Define $Q=(\mathbf I +R)/2$. Then, $Q$ is a projection $Q=Q^*=Q^2$ with $\mathbb E(Q) = \frac12 \mathbf I$. Suppose that for some $A\subset I$ we have
\[
||P_A Q P_A || \le \beta:=\frac{1+\ve}2.
\]
Since $Q$ is positive this can be phrased in terms of the partial order $\le$ on self-adjoint operators
\[
\mathbf 0 \le P_A Q P_A \le \beta P_A.
\]
Since $R=2Q-\mathbf I$ we obtain
\begin{equation}\label{pa}
-P_A \le P_A R P_A \le (2\beta-1) P_A =\ve P_A.
\end{equation}
We repeat the same for a projection $Q_1=(\mathbf I -R)/2$. Assuming that for some $B\subset I$ we have
\[
||P_B Q_1 P_B || \le \beta:=\frac{1+\ve}2,
\]
yields
\begin{equation}\label{pb}
-P_B \le P_B (-R) P_B \le (2\beta-1) P_B =\ve P_B.
\end{equation}
Taking $C=A\cap B$ and combining \eqref{pa} and \eqref{pb} yields
\begin{equation}\label{pc}
-\ve P_C  = - \ve P_C P_B P_C \le  P_C R P_C \le \ve P_C P_A P_C= \ve P_C.
\end{equation}
Hence, $||P_C R P_C|| \le \ve$. By the paving property $(PP_{\frac12})$, we can find partitions $\{A_1,\ldots, A_r\}$ and $\{B_1,\ldots,B_r\}$ which produce $(r,\frac{1+\ve}2)$-paving of $Q$ and $Q_1$, resp. By \eqref{pc} their common refinement partition $\{C_{i,j}=A_i \cap B_j\}_{i,j \in [r]}$ yields $(r^2,\ve)$-paving of $R$. 
\end{proof}

\subsection{Weaver's conjecture}
Next, we will show that paving conjectures follow from Weaver's $KS_r$ conjecture, which was verified by Marcus, Spielman, and Srivastava \cite{MSS}. We state it in a general infinite dimensional form $(KS^\infty_r)$ and later deduce it from its finite dimensional counterpart $(KS_r)$. We start with the standard definition in frame theory.

\begin{definition}\label{def1}
A family of vectors $\{u_i \}_{i\in I}$ in a Hilbert space
$\mathcal H$ is called a {\it frame} for $\mathcal H$ if
there are constants $0<A\le B < \infty$ (called 
{\it lower and upper frame bounds}, respectively)
so that 
\begin{equation}\label{E5}
A\|u\|^2 \le \sum_{i\in I}|\langle u,u_i \rangle |^2
\le
B\| u\|^2 \qquad\text{for all }u\in \mathcal H.
\end{equation}
If we only have the right hand inequality in \eqref{E5},
 we call $\{u_i\}_{i\in I}$
a {\it Bessel sequence} with Bessel bound $B$.
  If $A=B$, $\{u_i \}_{i\in I}$ is called a
{\it tight frame} and if $A=B=1$, it is called a
{\it Parseval frame}.
\end{definition}

\begin{theorem*}[$KS^\infty_r$]\label{MSS}
Let $I$ be at most countable index set and let $\mathcal H$ be a separable Hilbert space. Let $\{u_i\}_{i \in I} \subset \mathcal H$ be a Bessel sequence with bound $1$,
\begin{equation}\label{mss1}
\sum_{i\in I} |\langle u, u_i \rangle |^2 \le 1 \quad\text{for all }||u||=1 \qquad\text{and}\qquad\|u_i\|^2 \le \delta
\quad\text{for all }i.
\end{equation}
Then for any positive integer $r$, there exists a partition $\{I_1,\ldots, I_r\}$ of $I$ such that each $\{u_i\}_{i\in  I_k}$, $k=1,\ldots,r$, is a Bessel sequence with the following bound 
\begin{equation}\label{mss2}
\sum_{i\in I_k} |\langle u, u_i \rangle |^2 \le \left(\frac{1}{\sqrt{r}} + \sqrt{\delta}\right)^2  \qquad\text{for all }||u||=1.
\end{equation}
\end{theorem*}

Next we show how $(KS_r^\infty)$ implies projection paving.

\begin{lemma} \label{kspp}
$(KS_r^\infty) \implies (PP_{\frac12})$.
\end{lemma}

\begin{proof}
Let $Q$ be an arbitrary projection on $\ell^2(I)$ with $\mathbb E(Q)= \frac12 \mathbf I$. Define vectors $u_i=Qe_i$, $i\in I$, where $\{e_i\}_{i\in I}$ is a standard o.n. basis of $\ell^2(I)$. Then, $Q$ is represented by the Gram matrix of $\{u_i\}_{i \in I}$
\[
Q = ( \langle Qe_i, Qe_j \rangle )_{i,j \in I}= ( \langle u_i, u_j \rangle )_{i,j \in I} \qquad\text{and}\qquad ||u_i||^2 = \frac12 = \delta.
\]
The Gram matrix $Q=TT^*$ is a composition of the analysis operator 
\[
T: \mathcal H \to \ell^2(I)\qquad\text{where } Tu= (\langle u, u_i\rangle)_{i\in I} \quad\text{ for } u \in\mathcal H,
\]
with the synthesis operator
\[
T^*: \ell^2(I) \to \mathcal H \qquad\text{where } Ta= \sum_{i\in I} a_i u_i \quad\text{ for } a=(a_i)_{i\in I} \in\ell^2(I).
\]
The frame operator  is a composition of these operators, but in a reverse order
\[
S:  \mathcal H \to \mathcal H \qquad S=T^*T= \sum_{i\in I} u_i \otimes u_i ,
\]
where $u_i \otimes u_i: \mathcal H \to \mathcal H$ is a rank one positive operator given by  
\[
(u_i \otimes u_i)(u) = \langle u, u_i \rangle u_i \quad\text{ for } u \in\mathcal H.
\]
By $(KS_r^\infty)$, for any $r\in \N$, there exists a partition $\{I_1,\ldots, I_r\}$ such that
\[
||P_{I_k} Q P_{I_k}|| = ||( \langle u_i, u_j \rangle)_{i,j \in I_k} || = \bigg\| \sum_{i \in I_k} u_i \otimes u_i \bigg\| \le \bigg(\frac1{\sqrt{r}}+\frac1{\sqrt{2}}\bigg)^2 < \frac12 + \frac3{\sqrt{r}}.
\]
The second equality is the consequence of the fact the norms of Gram and frame operator are the same $||T T^*||=||T^* T||$. Thus, $Q$ can be $(r, \frac{1+\ve}2)$-paved for $r=36/\ve^2$.
\end{proof}

Marcus, Spielman, and Srivastava \cite{MSS} have shown the following version of Weaver's conjecture. The key feature of $(KS_r)$ is independence of Bessel bound on a number  of vectors $m$ and a dimension $d$.

\begin{theorem*}[$KS_r$]\label{MSS2}
Let $\{u_i\}_{i \in [m]} \subset \C^d$ be a Parseval frame
\begin{equation}\label{mss3}
\sum_{i=1}^m |\langle u, u_i \rangle |^2 = 1 \quad\text{for all }||u||=1 \qquad\text{and}\qquad\|u_i\|^2 \le \delta
\quad\text{for all }i.
\end{equation}
Then for any positive integer $r$, there exists a partition $\{I_1,\ldots, I_r\}$ of $[m]$ such that each $\{u_i\}_{i\in  I_k}$, $k=1,\ldots,r$, is a Bessel sequence with bound $ \left(\frac{1}{\sqrt{r}} + \sqrt{\delta}\right)^2$, i.e., \eqref{mss2} holds.
\end{theorem*}

To deduce Theorem $(KS^\infty_r)$ from $(KS_r)$ we will use the following fact, which is sometimes referred to as a pinball principle. Its proof is essentially a combination of diagonal argument with pigeonhole principle.

\begin{lemma}\label{pinball} 
Fix a natural number $r$ and assume for every natural number $n$, we have a partition $\{I_i^n\}_{i=1}^r$ of $[n]$. Then there are natural numbers $\{n_1 < n_2 < \cdots \}$ so that if $j \in I_i^{n_j}$ for some $i \in [r]$, then $j \in I_i^{n_k}$ for all $k \ge j$. For any $i \in [r]$ define $I_i = \{j: j\in I_i^{n_j}\}$. Then,
\begin{enumerate}[(i)]
\item $\{I_i\}_{i=1}^r$ is a partition of $\N$.
\item If $I_i = \{j_1 < j_2 < \cdots\}$, then for every natural number $k$ we have 
\[\{j_1,j_2,\ldots,j_k\} \subset I_i^{n_{j_k}}.\]
\end{enumerate}
\end{lemma}

Instead of giving a separate proof of Lemma \ref{pinball}, we include its justification in the proof of Lemma \ref{ksinf}.

\begin{lemma} \label{ksinf}
$(KS_r) \implies (KS^\infty_r)$.
\end{lemma}

\begin{proof}
First, observe that the Parseval frame assumption \eqref{mss3} can be weakened by the Bessel condition. Indeed, suppose that $\{u_i\}_{i\in [m]}$ is merely a Bessel sequence with bound $1$ and $||u_i||^2 \le \delta$. Define $d\times d$ matrix $T$ as
\[
T = \mathbf I - \sum_{i=1}^m u_i \otimes u_i.
\]
Since $T$ is positive semidefinite, we can find vectors $\{u_i\}_{i=m+1}^{m'}$, $m'>m$, such that
\[
T = \sum_{i=m+1}^{m'} u_i \otimes u_i \qquad\text{and}\qquad ||u_i||^2 \le \delta \text{ for }i\ge m+1.
\]
Indeed, it suffices to choose vectors $u_i$ to be appropriately scaled eigenvectors of $T$. Consequently, $\{u_i\}_{i\in [m']}$ becomes a Parseval frame for $\C^d$ and by $(KS_r)$ we can find a partition $\{I_1,\ldots, I_r\}$ of $[m']$ such that corresponding subsets $\{u_i\}_{i\in I_k}$ have required Bessel bounds. Restricting these partition to $[m]$ yields the same conclusion for $\{u_i\}_{i\in I_k\cap [m]}$, $k=1,\ldots,r$.

Now suppose $\{u_i\}_{i\in I}$ is an infinite Bessel sequence in a Hilbert space $\mathcal H$ as in \eqref{mss1}. Since $I$ is countable, we may assume $I=\N$. For any $n\in \N$ we can apply $(KS_r)$ to the initial sequence $\{u_i\}_{i\in [n]}$. Hence, for each $n\in \N$ we have a partition $\{I_1^n, \ldots, I_r^n\}$ of $[n]$, which yields required Bessel bounds. To show the existence of a global partition of $\{I_1,\ldots,I_r\}$ of $\N$ satisfying \eqref{mss2}, it suffices to apply Lemma \ref{pinball}. This boils down to repeated applications of pigeonhole principle. The first vector $u_1$ must land infinitely many times to one of the slots $I_{j_1}^n$ for some $j_1=1,\ldots,r$. Let $N_1 \subset \N$ be the collection of such $n$. Then, we repeat the same argument to the second vector $u_2$ for partitions of $[n]$, where $n\in N_1$. Again, we can find a slot $I_{j_2}^n$, where the second vector $u_2$ lands for infinitely many $n\in N_2 \subset N_1$. Repeating this process yields a nested sequence of infinite subsets $N_1 \supset N_2 \supset \ldots$ and indices $j_1,j_2,\ldots$ in $[r]$ such that the initial vectors $u_1,\ldots, u_m$, $m\in \N$, all land to the same respective slots $I^n_{j_1},\ldots,I^n_{j_m}$ for all $n\in N_m$. This yields a global partition of $\N$ by $I_k =\{ i \in \N: j_i=k\}$, $k\in [r]$. Thus, \eqref{mss2} holds when $I_k$ replaced by $I_k \cap [m]$. Letting $m\to \infty$ shows the required Bessel bound \eqref{mss2}.
\end{proof}

An interesting special case of Weaver's conjecture $(KS_r)$ happens when $r=2$.

\begin{theorem*}[$KS_2^\infty$]
Let $I$ be at most countable index set and let $\mathcal H$ be a separable Hilbert space. Let $\{u_i\}_{i \in I} \subset \mathcal H$ be a Parseval frame and $\|u_i\|^2 \le \delta$ for all i. Then, there exists a partition $\{I_1,I_2\}$ of $I$ such that each $\{u_i\}_{i\in  I_k}$, $k=1,2$ is a frame with bounds
\
\begin{equation}\label{mss4}
\frac12-O(\sqrt{\delta}) \le \sum_{i\in I_k} |\langle u, u_i \rangle |^2 \le \frac12 + O(\sqrt{\delta})  \qquad\text{for all }||u||=1.
\end{equation}
\end{theorem*}

\begin{lemma} $(KS^\infty_r) \implies (KS^\infty_2)$.
\end{lemma}

\begin{proof}
$(KS_r)$ for $r=2$ yields partition $\{I_1,I_2\}$ such that
\[
\sum_{i\in I_k} |\langle u, u_i \rangle |^2 \le \frac12 + \sqrt{2\delta} + \delta \qquad\text{for }||u||=1, \ k=1,2.
\]
Subtracting the equality \eqref{mss3} yields the lower bound in \eqref{mss4}.
\end{proof}

\begin{remark}\label{rem}
Note that the bound \eqref{mss4} produces something non-trivial only for $0<\delta<(2+\sqrt{2})^{-2} \approx 0.0857864$. Casazza, Marcus, Speegle, and the author \cite{BCMS} have shown the improved bound in $(KS_2)$. For $0<\delta<\frac14$, the bound \eqref{mss4} holds where $O(\sqrt{\delta})$ is replaced by $\sqrt{2\delta(1-2\delta)}$. Therefore, by a variant of Lemma \ref{kspp} we have the following variant of paving for projections.
\end{remark}

\begin{theorem*}[$PP_\delta$] Let $0<\delta<1/4$. Every projection $P$ on $\ell^2(I)$,  i.e., $P=P^*$ and $P^2=P$, with all diagonal entries $\le  \delta$ can be $(2,\frac12 + \sqrt{2\delta(1-2\delta)})$-paved.
\end{theorem*}

\subsection{Notes}
Another well-known equivalent of the Kadison-Singer problem, which we didn't discuss here, is a relative Dixmier property studied by Berman, Halpern, Kaftal, and Weiss \cite{BHKW, HKW, HKW2, HKW3}. Every bounded operator $T\in \mathcal B(\ell^2(I))$ satisfies
\[
\mathbb E(T) \in \overline{\operatorname{conv}} \{ U T U^*: U \text{ is a diagonal unitary on $\ell^2(I)$}\}.
\]
The connection of the Kadison-Singer problem with paving was investigated by Anderson and Akemann \cite{AA, An1, An2, An3}. A streamlined presentation of paving implications presented here has been shown by Casazza, Edidin, Kalra, and Paulsen \cite{CEKP}. A pinball principle, Lemma \ref{pinball}, was shown in \cite{CCLV}.

For each of the classes of matrices/operators considered above, such as:
\begin{itemize}
\item[$(B)$] bounded matrices with zero diagonal,
\item[$(S)$] self-adjoint matrices with zero diagonal,
\item[$(P_{\frac12})$] projections with $\frac12$ on diagonal,
\end{itemize}
we can ask for the smallest $r\in \N$ such that  all matrices in this class have $(r,\ve)$-paving for some $\ve<1$. By keeping track of the values in Lemma \ref{kspp}, we have shown $(12,\ve)$ paving for $(P_{\frac12})$ and thus $(12^2,\ve)$ paving for $(S)$ for some $\ve<1$. This was recently improved by Ravichandran \cite{Rav} who has shown $(4,\ve)$ paving for $(P_{\frac12})$ and thus $(16,\ve)$ paving for $(S)$. It is known that $(2,\ve)$ paving does not work for $(P_{\frac12})$, see \cite{CFMT}. Does $(3,\ve)$ paving work for $(P_{\frac12})$? Likewise, we can ask for largest $\delta$ such that $(2,\ve)$ paving works for all projections with $\delta$ on diagonal.
Paving property can be formulated for other operator norms and matrices with zero diagonal. However, paving remains an open problem for operator  $\ell^p$ norms, $p\ne 2$, though Schechtman \cite{Sch} has recently shown paving for the Schatten $C_p$ class norm for $2<p<\infty$ extending earlier results of Berman, Halpern, Kaftal, and Weiss \cite{BHKW2}.

\section{Proof of Weaver's conjecture}

Weaver's conjecture is a consequence of the following probabilistic result due to Marcus, Spielman, and Srivastava \cite{MSS}. The special case was shown by Casazza, Marcus, Speegle, and the author \cite{BCMS}.

\begin{theorem*}[$MSS$]\label{thmp}
Let $\epsilon > 0$. Suppose that $v_1, \dots, v_m$ are jointly independent random vectors in $\C^d$, which take finitely many values and satisfy  
\begin{equation}\label{thmp1}
\sum_{i=1}^m \E{}{v_{i} v_{i}^{*}} = \bI
\qquad\text{and}\qquad
\E{}{ \|v_{i}\|^{2}} \leq \epsilon \quad\text{for all } i.
\end{equation}
Then,
\begin{equation}\label{thmp2}
\P{}{\left\| \sum_{i=1}^m v_i v_i^{*} \right\| \leq (1 + \sqrt{\epsilon})^2} > 0.
\end{equation}
In the special case when  $v_1,\ldots,v_m$ take at most two values and $\epsilon<1/4$, we have
\[
\P{}{\left\| \sum_{i=1}^m v_i v_i^{*} \right\| \leq 1 + 2 \sqrt{\epsilon} \sqrt{1-\epsilon}} > 0.
\]
\end{theorem*}

\begin{lemma} $(MSS) \implies (KS_r) $.
\end{lemma}

\begin{proof} Assume $\{u_i\}_{i \in [m]} \subset \C^d$ satisfies \eqref{mss3}. For any $r\in \N$, let $v_1,\ldots,v_m$ be independent random vectors in $(\C^d)^{\oplus r}=\C^{rd}$ such that each vector $v_i$ takes $r$ values
\[
\begin{bmatrix} \sqrt{r}  u_i \\ 0 \\ \vdots  \\0 \end{bmatrix}, \ldots,
\begin{bmatrix}    0 \\ \vdots  \\ 0 \\ \sqrt{r}u_i \end{bmatrix}
\]
each with probability $\frac1r$. Then,
\[
\sum_{i=1}^m \E{}{v_i v_i^*} 
=  \begin{bmatrix} \sum_{i=1}^m u_i u_i^* &   &  \\   & \ddots &  \\  &  & \sum_{i=1}^m u_i u_i^*\end{bmatrix} = 
\begin{bmatrix} \mathbf I_d &   &  \\   & \ddots &  \\  &  & \mathbf I_d \end{bmatrix} =\mathbf I_{dr},
\]
and
\[
\E{}{||v_i||^2 } =  r||u_i||^2 \le \epsilon:=r\delta. 
\]
Hence, we can apply $(MSS)$ to deduce \eqref{thmp1}. Choose an outcome for which the bound in \eqref{thmp2} happens. For this outcome define
\[
I_k = \{i \in [m]: v_i \text{ is  non-zero in $k^{\rm th}$ entry} \}, \qquad\text{for } k=1,\ldots,r.
\]
Thus, the block diagonal matrix
\[
\sum_{i=1}^m v_i v_i^* 
= \begin{bmatrix} r \sum_{i\in I_1} u_i u_i^* &   &  \\   & \ddots &  \\  &  & r \sum_{i\in I_r} u_i u_i^*\end{bmatrix}
\]
has norm bounded by $(1 + \sqrt{\epsilon})^2$. This implies that each block has norm
\[
\bigg\| \sum_{i\in I_k} u_i u_i^* \bigg\| \le \frac1r (1+ \sqrt{r\delta})^2 =  \left(\frac{1}{\sqrt{r}} + \sqrt{\delta}\right)^2.
\]
Since a rank one operator $u_i \otimes u_i$ on $\C^d$ is represented by the $d\times d $ matrix $u_i u_i^*$, we obtain $(KS_r)$.
\end{proof}

\subsection{Mixed characteristic polynomial}
The main result of this section involves the concept of a mixed characteristic polynomial (MCP).

\begin{definition}
Let $A_1, \ldots, A_m$ be $d\times d$ matrices. The mixed characteristic polynomial is defined as for $z\in \C$ by
\[
\mu[A_1,\ldots,A_m](z) =  \bigg(\prod_{i=1}^m (1 - \partial_{z_i}) \bigg) \det \bigg( z \mathbf I + \sum_{i=1}^m z_i A_i \bigg)\bigg|_{z_1=\ldots=z_m=0}.
\]
\end{definition}

By determinant expansion one can show that $\det \bigg( z \mathbf I + \sum_{i=1}^m z_i A_i \bigg)$ is  a polynomial in $\C[z,z_1,\ldots,z_m]$ of degree $\le d$. Hence, $\mu[A_1,\ldots,A_m](z)$ is a polynomial in $\C[z]$ of degree $\le d$. These polynomials satisfy a number of interesting properties if $A_1, \ldots, A_m$ are positive definite. 

\begin{theorem*}[$MCP$]\label{mixed}
Let $\epsilon>0$. Suppose $A_1, \ldots, A_m$ are $d \times d$ positive semidefinite matrices satisfying
\begin{equation}\label{mixed1}
\sum_{i=1}^m A_i = \bI 
\qquad\text{and}\qquad
\operatorname{Tr}(A_i) \le \epsilon \quad\text{for all i}.
\end{equation}
Then, all roots of the mixed characteristic polynomial $\mu[A_1,\ldots, A_m]$ are real and the largest root
is at most $(1 + \sqrt{\epsilon})^2$.
\end{theorem*}

It remains to accomplish two major tasks: prove the implication $(MCP) \implies (MSS)$ and then show $(MCP)$. Before doing this we need to show a few basic properties of $\mu$.

\begin{lemma}\label{mas} For a fixed $z\in \C$, the mixed characteristic polynomial mapping 
\[
\mu: M_{d\times d}(\C) \times \ldots\times  M_{d \times d}(\C) \to \C
\]
 is multi-affine and symmetric. That is, $\mu$ affine in each variable and its value is the same  for any permutation of its arguments $A_1,\ldots,A_m$.
\end{lemma}

\begin{proof}
The fact the $\mu$ is symmetric is immediate from the definition. We claim that for any $d\times d$ matrix $B$, a function 
\[
f: M_{d\times d}(\C) \to \C, \qquad f(A_1) = (1 - \partial_{z_1}) \det( B+ z_1 A_1)|_{z_1=0} \qquad\text{for } A_1 \in M_{d\times d}(\C)
\]
is affine. Indeed, if $B$ is invertible, then by Jacobi's formula
\[
f(A_1)= \det(B)  - \det( B) \partial_{z_1}\det(\mathbf I + z_1 B^{-1}A_1)|_{z_1=0} = \det(B)(1 - \tr(B^{-1}A_1)).
\]
Since invertible matrices are dense in the set of all matrices, by continuity we deduce the general case. Thus, for any choice of matrices $A_2, \ldots, A_m$, a mapping
\[
(M_{d\times d}(\C),\C^{m-1}) \ni (A_1, z_2,\ldots,z_m) \mapsto  (1 - \partial_{z_1})\det \bigg( z \mathbf I + \sum_{i=1}^m z_i A_i \bigg) \bigg|_{z_1=0}
\]
is affine in the $A_1$ variable and a polynomial of degree $\le d$ in $z_2,\ldots,z_m$ variables. Applying linear operators, such as partial differential operators with constant coefficients $(1-\partial_{z_i})$, $i=2,\ldots,m$, preserves this property. Consequently, the mapping $A_1 \mapsto \mu[A_1,\ldots,A_m](z)$ is affine. By symmetry, $\mu$ is multi-affine. \end{proof}

\begin{lemma}\label{mu}
If $A_1, \ldots, A_m$ are rank one $d\times d$ matrices, then the mixed characteristic polynomial is a characteristic polynomial of the sum $A=A_1+\ldots+A_m$,
\begin{equation}\label{mud}
\mu[A_1,\ldots,A_m](z) = \det (z \mathbf I - A) \qquad z\in \C.
\end{equation}
\end{lemma}

\begin{proof}
Any rank one matrix is of the from $u v^*$ for some $u,v \in \C^d$. By the Sylvester determinant identity $\det(\mathbf I + t uv^*)=1+t v^*u$ for any $t\in \C$. Hence, for any $d\times d$ matrix $B$ the mapping
\[
\C \ni t \mapsto \det (B+ t uv^*) = b_0+ b_1 t , \qquad b_0,b_1 \in \C,
\]
is affine. If $B$ is invertible, then this follows by factoring out $B$, which reduces to the case $B=\mathbf I$. Since invertible matrices are dense in the set of all matrices, by continuity we deduce the general case. This implies that for fixed $z\in \C$, the polynomial
\[
p(z_1,\ldots,z_m):=  \det \bigg( z \mathbf I + \sum_{i=1}^m z_i A_i \bigg)= b + \sum_{1 \le i_1< \ldots<i_j \le m} a_{i_1} \ldots a_{i_j}z_{i_1} \ldots z_{i_j}
\]
is affine multilinear in $z_1, \ldots, z_m$. Hence, we can recover values of $p$ by the following formula
\[
p(t_1,\ldots, t_m) =  \bigg(\prod_{i=1}^m (1 + t_i \partial_{z_i}) \bigg)p(z_1,\ldots,z_m) \bigg|_{z_1=\ldots=z_m=0}.
\]
Taking $t_1=\ldots=t_m=-1$ yields
\[
\mu[A_1,\ldots,A_m](z) = p(-1,\ldots,-1)= \det (z \mathbf I - A).
\]
\end{proof}

\begin{lemma}\label{mux}
Let $X_1,\ldots ,X_m$ be $d\times d$ jointly independent random matrices, which take finitely many values. Then,
\begin{equation}\label{mux1}
\E{}{\mu[X_1,\ldots, X_m](z)} = \mu[\E{}{X_1}, \ldots, \E{}{X_m}](z) \qquad z\in\C.
\end{equation}
In addition, if random matrices $X_i$, $i=1,\ldots,m$, are rank one, then
\begin{equation}\label{mux2}
\E{}{\det\bigg(z\mathbf I -\sum_{i=1}^m X_i\bigg)} = \mu[\E{}{X_1}, \ldots, \E{}{X_m}](z) \qquad z\in\C.
\end{equation}
\end{lemma}

\begin{proof}
By Lemma \ref{mas} for any matrices $B_1,\ldots,B_n$ and $A_2,\ldots,A_m$ and coefficients $p_1,\ldots,p_n$ satisfying $\sum_{i=1}^n p_i=1$, we have
\[
\mu\bigg[\sum_{i=1}^n p_i B_i,A_2, \ldots, A_m\bigg ](z) = \sum_{i=1}^n p_i \mu[B_i,A_2,\ldots,A_m].
\]
Then the joint independence of $X_1,\ldots ,X_m$ yields \eqref{mux1}. Combining \eqref{mux1} with Lemma \ref{mu} yields \eqref{mux2}.
\end{proof}

\subsection{Real stable polynomials}
The proof of $(MCP) \implies (MSS)$ relies on the concept of a real stable polynomial.

\begin{definition}
Let $\C_+ = \{ z \in \C : \Im(z) > 0 \}$ be the upper half plane. 
We say that a polynomial $p \in \C[z_1, \dots, z_m]$ is {\em stable} if $p(z_1,\ldots,z_m) \neq 0$ for every $(z_1,\ldots,z_m) \in \C_+^m$.
A polynomial is called {\em real stable} if it is stable and all of its coefficients are real.
\end{definition}

Note that a univariate polynomial is real stable if and only if all its roots are real. We will show a few basic properties of real stable polynomials.

\begin{lemma}\label{rs}
If $A_1,\ldots,A_m$ are positive semidefinite hermitian $d\times d$ matrices, then
\[
p(z,z_1,\ldots,z_m) = \det \bigg( z \mathbf I + \sum_{i=1}^m z_i A_i \bigg) \in \C[z,z_1,\ldots,z_m]
\]
is a real stable polynomial.
\end{lemma}

\begin{proof}
If inputs $z,z_1,\ldots, z_m$ are real, then the values $p(z,z_1,\ldots,z_m)$ are also real, since  a determinant of a hermitian matrix is real. Hence, $p\in \R[z,z_1,\ldots,z_m]$. On the contrary suppose that $p(z,z_1,\ldots,z_m)=0$ for some $(z,z_1,\ldots,z_m) \in \C_+^{m+1}$. That is,
\[
\bigg( z \mathbf I + \sum_{i=1}^m z_i A_i \bigg)v=0 \qquad\text{for some }0 \ne v\in \C^{m+1}.
\]
Hence,
\[
0= \Im \bigg\lan \bigg( z \mathbf I + \sum_{i=1}^m z_i A_i \bigg)v,v \bigg \rangle = \Im(z) ||v||^2 + \sum_{i=1}^m \Im(z_i) \lan A_i v, v \ran \ge \Im(z) ||v||^2 >0,
\]
which is a contradiction.
\end{proof}

\begin{lemma}\label{rd}
 Suppose that $p\in \R[z_1,\ldots,z_m]$ is stable.
\begin{itemize}
\item
{\bf (restriction)} for fixed $t\in\R$, polynomial $p(t,z_2,\ldots,z_{m}) \in \R[z_2,\ldots,z_m]$ is stable unless it is identically zero.
\item {\bf (differentiation)} if $t\in \R$, then $(1+t\partial_{z_1})p $ is real stable.
\end {itemize}
\end{lemma}

\begin{proof} By Hurwitz's theorem, if a sequence of non-vanishing holomorphic functions $\{f_n\}_{n\in\N}$ on an open connected domain $\Omega \subset \C^m$ converges uniformly on compact sets, then its limit $f$ is either non-vanishing or $f \equiv 0$. Let $\Omega = \C_+^{m-1}$. Define
\[
f_n(z_1,\ldots,z_{m-1}) = p(t+i/n,z_1,\ldots,z_{m-1}) \qquad\text{for }(z_1,\ldots, z_{m-1}) \in \Omega, \ n\in\N.
\]
Letting $n\to \infty$, Hurwitz's theorem implies the restriction property.

To show differentiation property we can assume that $t \ne 0$. Fix $z_2,\ldots,z_m \in \Omega$. By definition $q(z)=p(z,z_2,\ldots,z_m) \in \C[z]$ is stable.  Hence, we can write $q(z)= c \prod_{i=1}^d (z-w_i)$ for some roots $w_1,\ldots,w_d \in \C \setminus \C_+$. Then, 
\[
q(z)+t q'(z) = c \prod_{i=1}^d (z-w_i) \bigg( 1+ \sum_{i=1}^d \frac{t}{z-w_i} \bigg).
\]
Take any $z\in \C_+$. Since $\Im(w_i) \le 0$, we have $z-w_i \in \C_+$, and hence $\Im (1/(z-w_i))<0$ for all $i=1,\ldots,d$. Hence, $\sum_{i=1}^d \frac{t}{z-w_i}$ has non-zero imaginary part. This implies that $q(z)+tq'(z) \ne 0$ for any $z\in \C_+$. Since $z_2,\ldots,z_m \in \Omega$ is arbitrary, $(1+t\partial_{z_1})p $ is stable.
\end{proof}

As a corollary of Lemma \ref{rs} and \ref{rd} we have:

\begin{corollary}\label{crs}
If $A_1,\ldots,A_m$ are positive semidefinite hermitian $d\times d$ matrices, then the mixed characteristic polynomial $\mu[A_1,\ldots,A_m]$ is real, stable, and monic of degree $d$.
\end{corollary}

The following elementary lemma plays a key role in our arguments. Recall that for any $p\in \R[z]$, $p$ is stable $\iff$ $p$ has all real roots. Let $\maxroot(p)$ be the largest root of $p$.

\begin{lemma}\label{el}
Let $p, q \in \R[z]$ be stable monic polynomials of the same degree. Suppose that every convex combination $(1-t)p+tq$, $0\le t \le 1$, is also stable. Then for any $0\le t_0 \le 1$, $\maxroot(((1-t_0)p+t_0q)$ lies between $\maxroot(p)$ and $\maxroot(q)$.
\end{lemma}

\begin{proof}
Without loss of generality we can assume  $\maxroot(p) \le \maxroot(q)$ and $0<t_0<1$. Our goal is to show that
\[
m_p:=\maxroot(p) \le \maxroot(((1-t_0)p+t_0q) \le \maxroot(q)=:m_q.
\]
For $x>m_q$, both $p(x)$ and $q(x)$ are positive, and hence $((1-t_0)p+t_0q)(x)>0$. This shows the second inequality.

We shall prove the first inequality by contradiction. Suppose that $ (1-t_0)p+t_0q$ has no roots $[m_p,m_q]$. This implies that $  (1-t_0)p+t_0q>0$ for all $x\ge m_p$. In particular, $q(m_p)>0$. Hence, $q$ must have at least $2$ roots (counting multiplicity) to the right of $m_p$. 
Let $D$ be an open disk in $\C$ centered at $\frac{m_p+m_q}2$ and radius $\frac{m_q-m_p}2$. We claim that
\[
((1-t)p+ t q)(z) \ne 0 \qquad \text{for all }z\in \partial D \text{ and }t_0 \le t \le 1.
\]
Indeed, since $(1-t)p+tq$ is stable, this is easily verified at $z=m_p$ and $z=m_q$. By compactness
\[
\inf_{(z,t) \in \partial D \times [t_0,1]} |((1-t)p+ t q)(z)|>0.
\]
By Rouche's theorem, polynomials $(1-t)p+tq$ have the same number of zeros in $D$ for all $t_0\le t\le 1$. This is a contradiction with the hypothesis that $ (1-t_0)p+t_0q$ has no roots in $D$, but $q$ has at least 2 roots in $D$.
\end{proof}

Lemma \ref{el} can be generalized to control other roots, such as second largest, third largest, etc. This leads to a concept of an interlacing family of polynomials, which plays a fundamental role in the arguments of Marcus-Spielman-Srivastava \cite{MSS0, MSS}. 

\subsection{Interlacing family of polynomials}  We shall not give the formal definition of this concept. Instead, following Tao \cite{Tao} we will use the following lemma.

\begin{lemma}\label{ell}
Let $p_1,\ldots, p_n \in \R[z]$ be stable monic polynomials of the same degree. Suppose that every convex combination \[
\sum_{i=1}^n t_i p_i,\qquad{where } \ \sum_{i=1}^n t_i =1,\ t_i\ge 0
\]
is  a stable polynomial. Then, for any such convex combination there exists $1\le i_0 \le n$ such that
\begin{equation}\label{ell1}
\maxroot(p_{i_0}) \le \maxroot\bigg(\sum_{i=1}^n t_i p_i \bigg).
\end{equation}
\end{lemma}

\begin{proof}
Using Lemma \ref{el} we can easily show \eqref{ell1} by induction on the number of polynomials $p_1,\ldots,p_n$.
\end{proof}

\begin{lemma}\label{ind}
 Let $X$ be a random rank one positive semidefinite $d\times d$ matrix. Let $A_1,\ldots, A_m$ be $d\times d$ deterministic positive semidefinite matrices. Then with positive probability we have
\begin{equation}\label{ind1}
\maxroot(\mu[X,A_1,\ldots,A_m]) \le \maxroot(\mu[\E{}{X},A_1,\ldots,A_m]).
\end{equation}
\end{lemma}

\begin{proof}
Suppose that $X$ takes values $B_1,\ldots, B_n$. Define polynomials $p_i=\mu[B_i,A_1,\ldots,A_m]$, $i=1,\ldots,m$. By Corollary \ref{crs} these are real, stable, and  monic polynomials of degree $d$. Moreover, any convex combination $\sum_{i=1}^n t_i p_i$ is also stable. Indeed, consider a random variable $Y$ taking values $B_1,\ldots, B_n$ with probabilities $t_1,\ldots, t_n$, resp. By Lemma \ref{mux} and Corollary \ref{crs}
\[
\sum_{i=1}^n t_i p_i = \E{}{\mu[Y,A_1,\ldots, A_m]} = \mu[\E{}{Y},A_1,\ldots, A_m]
\]
is a real stable polynomial. Hence, Lemma \ref{ell} yields \eqref{ind1}.
\end{proof}

Iterating Lemma \ref{ind} gives the required control on roots of mixed characteristic polynomials. This is an essence of  the method of interlacing family of polynomials.

\begin{lemma}\label{max}
Suppose that $X_1, \dots, X_m$ are jointly independent random rank one positive semidefinite $d\times d$ matrices which take finitely many values. Then with positive probability
\begin{equation}\label{max1}
\maxroot(\mu[X_1,\ldots X_m]) \le  \maxroot(\mu[\E{}{X_1}, \ldots, \E{}{X_m}]).
\end{equation}
\end{lemma}

\begin{proof}
By Lemma \ref{ind} a random matrix $X_1$ takes  some value $A_1$ (with positive probability) such that
\[
\maxroot(\mu[\E{}{X_1}, \ldots, \E{}{X_m}]) \ge \maxroot(\mu[A_1,\E{}{X_2}, \ldots, \E{}{X_m}]).
\]
By the independence assumption, if we condition the probability space to the event $X_1=A_1$, then random variables $X_2,\ldots, X_m$ have the same joint distribution. Again by Lemma \ref{ind}, $X_2$ takes some value $A_2$ (with positive probability) such that
\[
\maxroot(\mu[A_1,\E{}{X_2}, \ldots, \E{}{X_m}]) \ge \maxroot(\mu[A_1,A_2,\E{}{X_3}, \ldots, \E{}{X_m}]).
\]
Conditioning on the event $X_1=A_1$ and $X_2=A_2$, and repeating this argument for remaining random variables yields \eqref{max1}. 
\end{proof}

Finally, we can complete the proof of the main implication.

\begin{lemma} $(MCP) \implies (MSS)$.
\end{lemma}

\begin{proof}
Take random vectors $v_1,\ldots, v_m$ as in Theorem $(MSS)$. Define rank one positive semidefinite random matrices $X_i=v_i v_i^*$, $i=1,\ldots,m$. The assumption \eqref{thmp1} translates into the assumption \eqref{mixed1} for $A_i=\E{}{X_i}$. Since $X_i$ are hermitian, by Lemma \ref{mu}
\[
\left\| \sum_{i=1}^m v_i v_i^{*} \right\| = \left\| \sum_{i=1}^m X_i \right\| = \maxroot \bigg(\det\bigg(z\mathbf I -\sum_{i=1}^m X_i\bigg) \bigg) = \maxroot(\mu[X_1,\ldots X_m]).
\]
By Lemma \ref{max}, the bound on $\maxroot(\mu[A_1,\ldots A_m])$ in the conclusion of Theorem $(MCP)$ yields the same bound on $\left\| \sum_{i=1}^m v_i v_i^{*} \right\|$ with positive probability.
\end{proof}

\subsection{Multivariate barrier argument} The proof of Theorem $(MCP)$ hinges on a multivariate barrier argument.

\begin{definition} Let $p\in \R[z_1,\ldots,z_m]$. We say that $x=(x_1,\ldots,x_m) \in \R^m$ is {\it above the roots} of $p$ if
\[
p(x+t)>0 \qquad\text{for all }t \in [0,\infty)^m.
\]
A {\it barrier function} of $p$ in direction of $z_i$, $i=1,\ldots,m$, is defined for such $x$ as
\[
\Phi^i_p(x) = \partial_{z_i} \log p(x)=  \frac{\partial_{z_i} p(x)}{p(x)}.
\]
\end{definition}

We need the following result about zeros of real stable polynomials in two variables, which is illustrated in Figure 1.

\begin{figure}\label{figures}
\begin{center}
\includegraphics[width=2.5in]{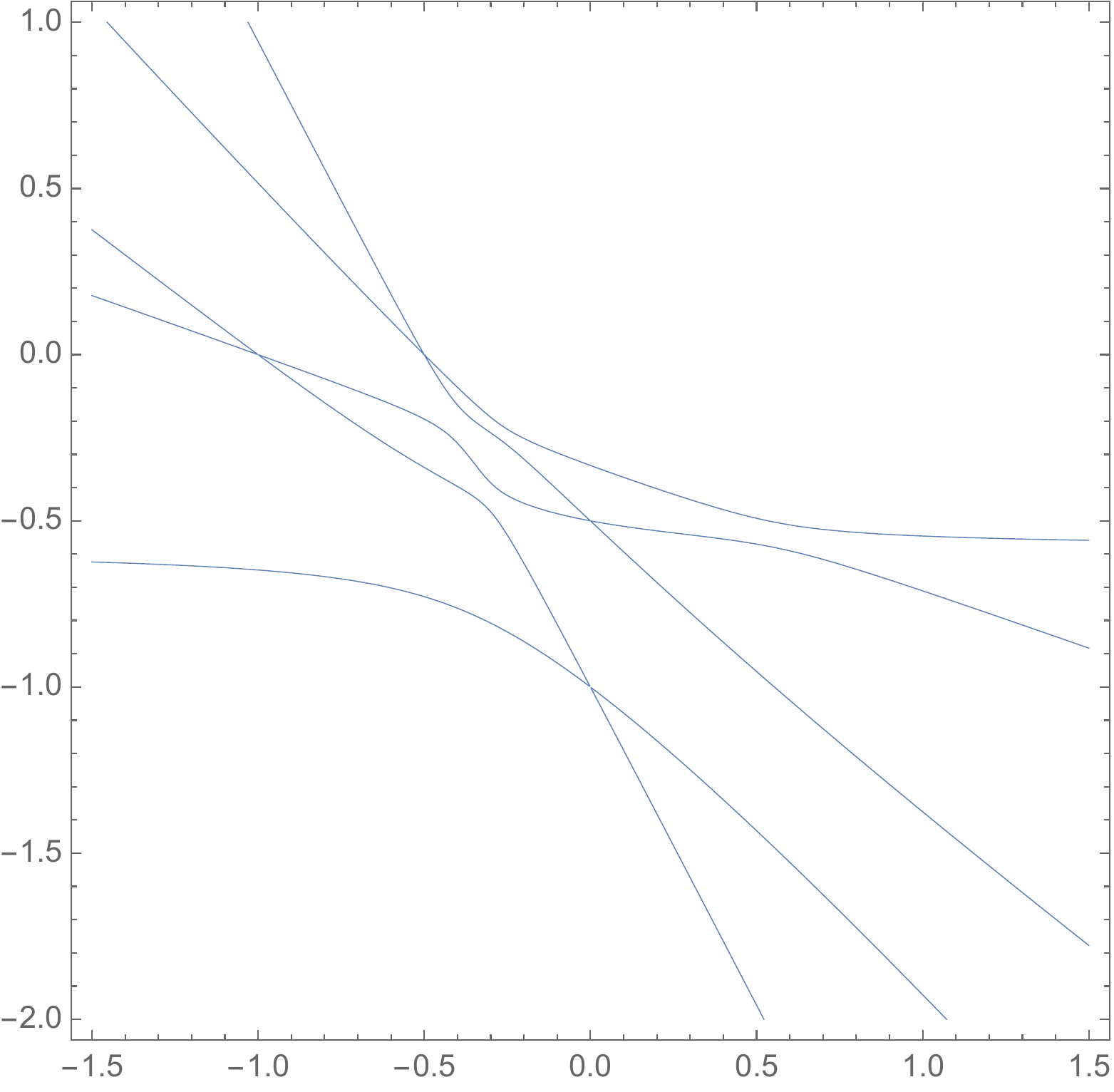}
\break
\includegraphics[width=2.5in]{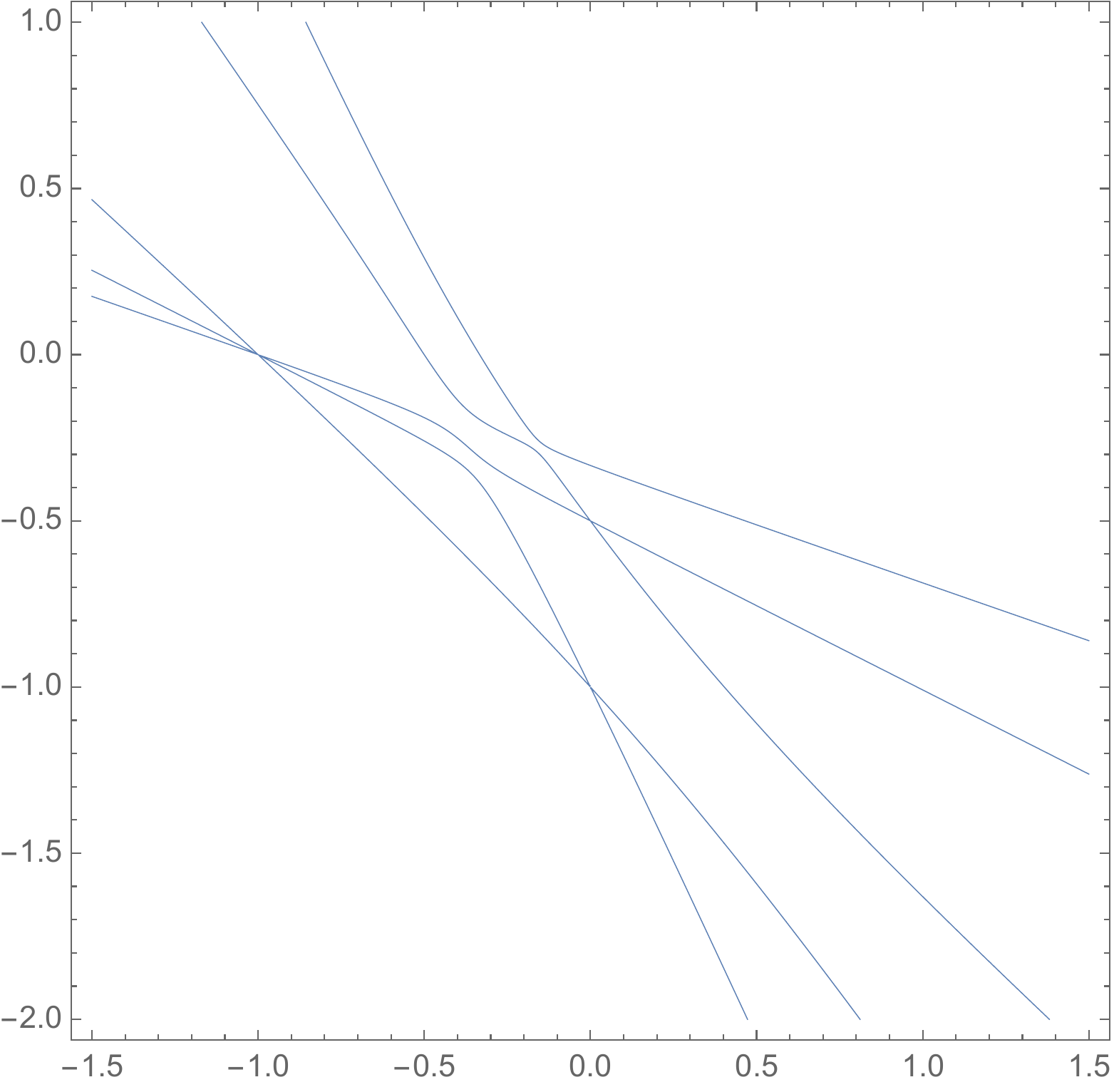}
\includegraphics[width=2.5in]{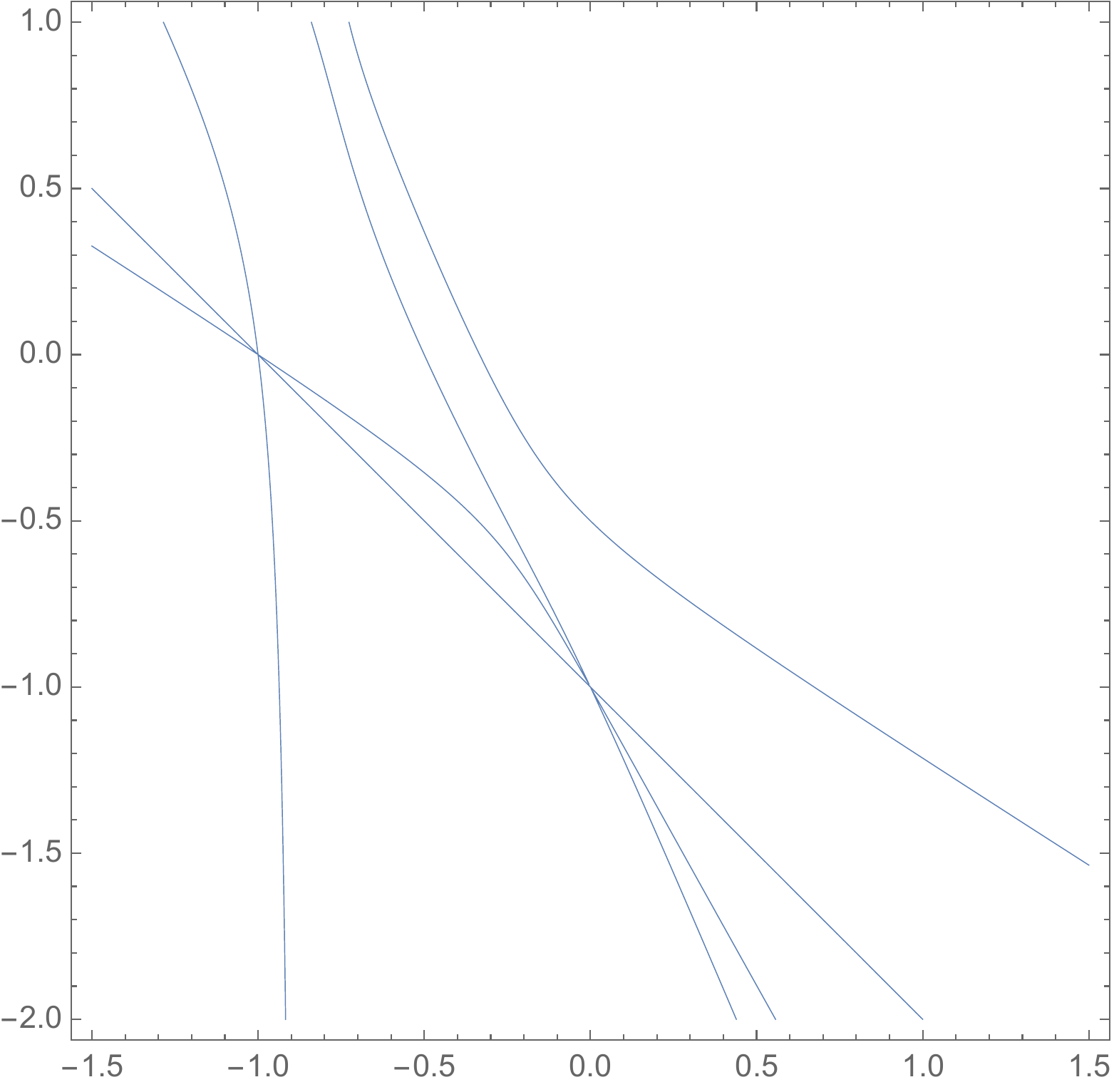}
\end{center}
\caption{Examples of zero sets of real stable polynomials on the plane.}
\end{figure}

\begin{lemma}\label{zp}
Let $p \in \R[z,w]$ be a stable polynomial. Then for all but finitely many $x \in \R$, a polynomial $p(x,w) \in \R[w]$ has all real roots and constant degree $d\in \N$. Let $y_1(x) \le \ldots \le y_d(x)$ be its roots counting multiplicity.  Then, for each $i\in [d]$, $x \mapsto y_i(x)$ is non-increasing.
\end{lemma}

\begin{proof}
We can write $p(z,w)= \sum_{i=0}^d w^i q_i(z)$, where each $q_i \in \R[z]$. By Lemma \ref{rd} for fixed $x \in \R$, a polynomial $p(x,w) \in \R[w]$ has all real roots. Its degree equals $d$ if and only if $q_d(x) \ne 0$. The fundamental theorem of algebra implies the existence of roots $y_1(x) \le \ldots \le y_d(x)$. It remains to show that $x \mapsto y_i(x)$ is non-increasing, where $i\in [d]$.

We claim that for every real root  $p(x,y)=0$,  $(x,y) \in \R^2$, we have
\begin{equation}\label{der}
\partial_{z} p(x,y)\le 0 \qquad\text{and}\qquad \partial_{w} p(x,y)\le 0.
\end{equation}
On the contrary, suppose that $\alpha = \partial_{w} p(x,y)>0$. By the implicit function theorem for holomorphic functions \cite[Theorem I.7.6]{FG}, there exists complex neighborhoods $U_{x}, U_{y} \subset \C$ of $x$ and $y$, resp., and a holomorphic function $h: U_{x} \to U_{y}$ such that
\[
\{(z,w) \in U_{x} \times U_{y}: p(z,w)=0 \} = \{ (z, h(z)): z \in U_{x} \}.
\]
Taking $z=x+ \ve i$ and $h(z) \approx y+ h'(x)\ve i=y+\alpha \ve i$ for small $\ve>0$ produces a root of $p$ with positive imaginary parts, which contradicts stability of $p$. By symmetry we deduce \eqref{der}.

To finish the proof, it is convenient to use a basic fact about algebraic curves
\[
\{ (x,y) \in \R^2: p(x,y) =0 \}.
\]
Every algebraic curve decomposes as a finite union of {\it branches} connected by some points and  a finite number of vertical lines. A branch is the graph of a smooth and monotone function defined on an open (possibly unbounded) interval in the $x$-axis. Hence, a branch is the graph of some $y_i(x)$ restricted to an appropriate open interval. Differentiating $p(x,y_i(x))=0$ with respect to $x$ yields
\[
\partial_{z} p(x,y_i(x)) + \partial_{w} p(x,y_i(x)) y_i'(x) =0.
\]
By \eqref{der} we have $y_i'(x) \le 0$.
\end{proof}

\begin{lemma}\label{bar}
Let $p \in \R[z_1,\ldots,z_m]$ be stable. Let $1\le i,j \le m$. Then, for any $k\in \N_0$, partial derivatives of the barrier function of $p$ satisfy
\[
(-1)^k \partial_{z_j}^k \Phi_p^i (x) \ge 0 \qquad\text{if }x\in \R^n \text{ is above the roots of $p$}.
\]
In particular, $t\mapsto \Phi_p^i(x+te_j)$ is non-negative, non-increasing, and convex function of $t\ge 0$.
\end{lemma}

\begin{proof}
First suppose that $i=j$. Freezing all variables except $z_i=z_j$, by Lemma \ref{rd} (restriction) we can assume that $m=1$. Suppose that $x\in \R$ is above all roots of a stable polynomial $p \in \R[z]$. The stability of $p$ implies that $p$ has all real roots. Hence, $p(z)=c \prod_{j=1}^d (z- y_j)$, where $y_j \in \R$, and
\[
\Phi_p(x) = \bigg(\frac{p'}p\bigg) (x) = \sum_{j=1}^d \frac{1}{x-y_j}.
\]
By a direct calculation
\[
(-1)^k (\Phi_p)^{(k)}(x) =  k! \sum_{j=1}^d \frac{1}{(x -y_j)^{k+1}}.
\]
Since $x$ is above the roots of $p$, we have $x> \max(y_j)$ and all of the above terms are positive.
The above argument also covers trivially the case $k=0$.

It remains to deal with the case $i \ne j$ and $k\ge 1$. By Lemma \ref{rd} and symmetry, we can assume that $m=2$, $i=1$, and $j=2$. Suppose that  $x=(x_1,x_2)$ is above the roots of a stable polynomial $p\in \R[z_1,z_2]$. Since
\[
(-1)^k \partial_{z_2}^k \Phi_p^1 (x) = (-1)^k \partial_{z_2}^k\partial_{z_1} \log p (x)= \partial_{z_1} ( (-1)^k \partial_{z_2}^k \log p)(x)
\]
it suffices to show that $x_1 \mapsto (-1)^k \partial_{z_2}^k \log p(x_1,x_2)$ is non-decreasing.

By Lemma \ref{zp} we can write
\[
p(x_1,x_2)=c(x_1) \prod_{i=1}^d (x_2-y_i(x_1)).
\]
Hence, 
\[
x_1 \mapsto (-1)^k\partial^k_{z_2} \log p(x_1,x_2) = - (k-1)! \sum_{i=1}^d \frac{1}{(x_2-y_i(x_1))^{k}}.
\]
Since $x$ is above the roots of $p$, we have $x_2> \max(y_i(x_1))$ and by  Lemma \ref{zp}, the above function is non-decreasing.
\end{proof}

The following lemma provides the crucial control of the barrier function of $(1-\partial_{z_j})q$ in terms of the barrier function of $q$.

\begin{lemma}\label{br2} 
 Let $q\in \R[z_1,\ldots,z_m]$ be stable. Suppose that $x\in \R^m$ lies above the roots of $q$ and 
\begin{equation}\label{br3}
\Phi^j_q(x)\le 1-\frac{1}\delta \qquad\text{for some }j\in [m] \text{ and } \delta>0.
\end{equation}
Then, $x+\delta e_j$ lies above the roots of $(1-\partial_{z_j})q$ and
\begin{equation}\label{br4}
\Phi^i_{(1-\partial_{z_j})q} (x+\delta e_j) \le \Phi^i_q(x) \qquad\text{for all }i\in [m].
\end{equation}
\end{lemma}

\begin{proof}
Take any $y\in \R^m$ above $x$, that is, $y_i \ge x_i$ for all $i\in [m]$. By Lemma \ref{bar} (monotonicity), we have  $\Phi^j_q(y) \le \Phi^j_q(x)<1$. Hence,
\begin{equation}\label{br5}
(1-\partial_{z_j})q(y) = q(y)(1- \Phi^j_q(y)) >0.
\end{equation}
In particular, $x+ \delta e_j$ is above the roots of $(1-\partial_{z_j})q$. By \eqref{br5}
\[
\log ((1-\partial_{z_j})q)(y) = \log q(y) + \log(1-\Phi^j_q)(y).
\]
Applying $\partial_{z_i}$, $i\in [m]$, shows
\[
\Phi^i_{(1-\partial_{z_j})q}(y) = \Phi^i_q(y) - \frac{\partial_{z_i}\Phi^j_q(y)}{1-\Phi^j_q(y)}.
\]
Since 
\[
\partial_{z_i}\Phi^j_q(y) = \partial_{z_i}\partial_{z_j} \log q(y) =\partial_{z_j}\partial_{z_i} \log q(y)= \partial_{z_j} \Phi^i_q(y),
\]
the required bound \eqref{br4} is equivalent with the inequality
\begin{equation}\label{br6}
 \frac{\partial_{z_j}\Phi^i_q(x+\delta e_j)}{1-\Phi^j_q(x+\delta e_j)} = 
 \Phi^i_q(x+\delta e_j)- \Phi^i_{(1-\partial_{z_j})q}(x+\delta e_j)
  \ge  \Phi^i_q(x+\delta e_j)- \Phi^i_q(x).
 \end{equation}
By Lemma \ref{bar} (convexity) and (monotonicity), we have
\[
 \Phi^i_q(x+\delta e_j) - \Phi^i_q(x) \le \delta \partial_{z_j} \Phi^i_q(x+\delta e_j) \le 0.
\]
Hence, \eqref{br6} is implied by multiplying the inequality
\begin{equation}\label{br7}
\frac{1}{1-\Phi^j_q(x+\delta e_j)} \le \delta
\end{equation}
 by $\partial_{z_j} \Phi^i_q(x+\delta e_j)$.
Finally, \eqref{br7} holds true as a consequence of Lemma \ref{bar} (monotonicity) and \eqref{br3}
\[
\Phi^j_q(x+\delta e_j) \le \Phi^j_q(x) \le 1 - \frac{1}{\delta}.
\]
This shows \eqref{br6} and consequently \eqref{br4}.
\end{proof}

Applying inductively Lemma \ref{br2} yields the crucial corollary.

\begin{corollary} \label{brc}
Let $q\in \R[z_1,\ldots,z_m]$ be stable. Suppose that $x\in \R^m$ lies above the roots of $q$ and for some $\delta>0$ we have 
\[
\Phi^j_q(x)\le 1-\frac{1}\delta \qquad\text{for all }j\in [m].
\]
Then, $x+(\delta,\ldots,\delta)$ lies above the roots of $\prod_{i=1}^m(1-\partial_{z_i})q$.
\end{corollary}

\begin{proof} For $k=0,\ldots, m$, define
\[
y_k = x+ \delta \sum_{i=1}^k e_i, \qquad q_k = \prod_{i=1}^k (1-\partial_{z_i}) q.
\]
Then, using Lemma \ref{br2}, we show inductively that $y_k\in \R^m$ lies above the roots of $q_k$ for all $k\in [m]$.
\end{proof}

Finally, we are ready to give the proof of $(MCP)$.

\begin{proof}[Proof of Theorem $(MCP)$]
Define
\begin{equation}\label{p}
p(z_1,\ldots,z_m) = \det \bigg( \sum_{i=1}^m z_i A_i \bigg) \in \R[z_1,\ldots,z_m]
\end{equation}
By Lemmas \ref{rs} and \ref{rd}, $p$ is a real stable polynomial. By Jacobi's formula for any $j\in [m]$,
\[
\begin{aligned}
\partial_{z_j} p(x_1,\ldots,x_m) & = \partial_t \det \bigg( \sum_{i=1}^m x_i A_i + t A_j \bigg)\bigg|_{t=0}
\\
& =
\det \bigg( \sum_{i=1}^m x_i A_i \bigg)  \tr\bigg( \bigg(\sum_{i=1}^m x_i A_i \bigg)^{-1} A_j \bigg).
\end{aligned}
\]
Hence, by \eqref{mixed1}
\[
\Phi^j_p(t,\ldots,t)  = \tr( t^{-1} A_j) \le \frac{\epsilon}t 
\qquad\text{for } t>0.
\]
Moreover, $x=(t,\ldots,t)$ lies above the roots of $p$ for any $t>0$.
Take any $t, \delta>0$ such that
\begin{equation}\label{ept}
\frac{\epsilon}t + \frac 1\delta \le 1.
\end{equation}
By Corollary \ref{brc}, $(t+\delta,\ldots,t+\delta)$ lies above the roots of $\prod_{i=1}^m(1-\partial_{z_i})p$. Since
\[
\prod_{i=1}^m(1-\partial_{z_i})p(z,\ldots,z) = \mu[A_1,\ldots, A_m](z) \qquad\text{for any }z\in\C,
\]
the largest root of $\mu[A_1,\ldots,A_m]$ is $\le t+\delta$. Minimizing $t+\delta$ under the constraint \eqref{ept} yields 
$t= \sqrt{\epsilon}+\epsilon$ and $\delta=1+\sqrt{\epsilon}$. Hence, the largest root of $\mu[A_1,\ldots,A_m]$ is bounded by $(1+\epsilon)^2$.
\end{proof}


\subsection{Notes}
The strategy of the proofs of $(MSS)$ and $(MCP)$ follows the original proof of Marcus, Spielman, and Srivastava \cite{MSS} with strong influence by Tao's blog article \cite{Tao}. The main difference is in the proof of Lemma \ref{bar}. Tao uses more elementary properties of real stable polynomials in the form of Lemma \ref{zp}, whereas the original proof uses Helton-Vinnikov's theorem \cite{BB, HV}. This result states that every real stable polynomial in two variables of degree $d$ has a determinantal representation $p(x,y)=\pm \det( xA +y B +C)$ for some $d \times d$ positive semidefinite matrices $A,B$ and a symmetric matrix $C$.

The proof of the special case of Theorem $(MSS)$, for random variables taking at most two values, is more technical and it can be found in \cite{BCMS}. It relies on a variant of Theorem $(MCP)$ for matrices $A_1, \ldots, A_m$ of rank $\le 2$. This corresponds to a determinantal polynomial \eqref{p} which is quadratic with respect to each variable $z_1, \ldots, z_m$. Amazingly, such deceptively simple polynomial encodes all the information about roots of the mixed characteristic polynomial $\mu[A_1,\ldots,A_m]$, which is needed for showing $(MCP)$.

\section{Applications of Weaver's conjecture}

In this section we show applications of the solution of Kadison-Singer problem which are outside of the main sequence of implications $(KS) \Leftarrow \ldots \Leftarrow (MCP)$. Our main goal is to show quantitative bounds in Feichtinger's conjecture. To achieve this we need some background about Naimark's dilation theorem.

\subsection{Naimark's complements of frame partitions} We start with well-known Naimark's dilation theorem.

\begin{lemma}\label{nai}
 Let $\{u_i\}_{i\in I}$ be a Parseval frame in a Hilbert space $\mathcal H$. Then there exists a larger Hilbert space $\mathcal K \supset \mathcal H$ and an o.n. basis $\{e_i\}_{i\in I} \subset \mathcal K$ such that
\[
u_i =Pe_i
\qquad\text{for all } i \in I, \text{ where $P$ is an orthogonal projection of $\mathcal K$ onto $\mathcal H$.}
\]
Conversely, if $P$ is a projection of $\mathcal K$ onto a closed subspace $\mathcal H$, then $\{P e_i\}_{i\in I}$ is a Parseval frame in $\mathcal H$.
\end{lemma}

\begin{proof} Consider the analysis $T: \mathcal H \to \ell^2(I)$ as in Lemma  \ref{kspp}. Since $\{u_i\}_{i\in I}$ is a Parseval frame, $T$ is an isometry of $\mathcal H$ onto $T(\mathcal H) \subset \ell^2(I)$. Let $Q$ be the orthogonal projection of $\ell^2(I)$ onto $T(\mathcal H)$. Let $\{e_i\}_{i\in I}$ be the standard o.n. basis of $\ell^2(I)$. Since $T$ is an isometry, it suffices to show the conclusion for Parseval frame $\{T u_i\}_{i\in I}$ in $T(\mathcal H)$. In turn, this is a consequence of the following calculation. Since $\{T u_i\}_{i\in I}$ is a Parseval frame
\[
Qa= \sum_{i\in I} \langle a, T u_i \rangle T u_i \qquad\text{for all }a \in \ell^2(I).
\]
Thus, for any $i\in I_0$,
\[
Q e_{i_0} = \sum_{i\in I} \lan e_{i_0}, T u_i \ran T u_i = \sum_{i\in I} \overline{ \lan u_i, u_{i_0} \ran} T u_i = T \bigg( \sum_{i\in I} \lan u_{i_0}, u_i \ran u_i \bigg) = T u_{i_0}.
\]
\end{proof}

Lemma \ref{nai} leads to the concept of a Naimark's complement. This is a Parseval frame $\{(\mathbf I - P)e_i \}_{i\in I}$ in $\mathcal K \ominus \mathcal H$, where $\mathbf I$ is the identity on $\mathcal K \cong \ell^2(I)$.
Recall the definition of a Riesz sequence.

\begin{definition}\label{def2}
 A family of vectors $\{u_i\}_{i\in I}$ in a Hilbert space $\mathcal H$
is a {\it Riesz sequence} if there are constants $A,B>0$ so that
for all $\{a_i\}\in \ell^2(I)$ we have
\begin{equation}\label{R7}
A\sum_{i\in I}|a_i |^2 \le \bigg\|\sum_{i\in I}a_i u_i \bigg\|^2 \le
B\sum_{i\in I}|a_i|^2.
\end{equation}
We call $A,B$ {\it lower and upper Riesz
bounds} for $\{u_i\}_{i\in I}$.
\end{definition}

Note that it suffices to verify \eqref{R7} only for sequences $\{a_i\}$ with finitely many non-zero coefficients, since a standard convergence argument yields the same bounds \eqref{R7} for all infinitely supported sequences $\{a_i\}\in \ell^2(I)$.
In general we do not require that frame, Bessel, and Riesz bounds in Definitions \ref{def1} and \ref{def2} are optimal. In particular, a Bessel sequence with bound $B$ is automatically a Bessel sequence with bound $B' \ge B$.

\begin{lemma}\label{ce} Let $P:\ell^2(I) \to \ell^2(I)$ be the orthogonal projection onto a closed subspace $\mathcal H \subset \ell^2(I)$. Then, for any subset $J \subset I$ and $\delta>0$, the following are equivalent:
\begin{enumerate}[(i)]
\item
$\{ P e_i \}_{i \in J}$ is a Bessel sequence with bound $1-\delta$,
\item
$\{ (\mathbf I - P) e_i \}_{i \in J}$ is a Riesz sequence with lower bound $\delta$, where $\mathbf I$ is the identity on $\ell^2(I)$.
\end{enumerate}
\end{lemma}

\begin{proof}
Note that for any sequence of coefficients $\{a_i\} \in \ell^2(J)$,
\begin{equation}\label{ce3}
\sum_{i\in J} |a_i|^2 = \bigg\| \sum_{i\in J} a_i P e_i \bigg\|^2 + \bigg\|\sum_{i\in J} a_i (\mathbf I - P) e_i \bigg\|^2.
\end{equation}
Thus,
\begin{equation}\label{ce4}
\bigg\| \sum_{i\in J} a_i P e_i \bigg\|^2 \le (1-\delta) \sum_{i\in J} |a_i|^2 
\iff
\bigg\| \sum_{i\in J} a_i (\mathbf I - P) e_i \bigg\|^2 \ge \delta \sum_{i\in J} |a_i|^2.
\end{equation}
Observe that the inequality in the left hand side of \eqref{ce4} is equivalent to (i). This follows from the well-known fact that $||T||=||T^*||$, where $T$ is the analysis operator
\[
T: \mathcal H \to \ell^2(I), \qquad T \phi = \{\lan u, P e_i \ran\}_{i\in J}, \quad u \in \mathcal H,
\]
and its adjoint is the synthesis operator
\[
T^*: \ell^2(I) \to \mathcal H, \qquad T^*(\{a_i\}_{i\in J}) = \sum_{i\in J} a_i Pe_i, \quad \{a_i\}_{i\in J} \in \ell^2(J).
\]
This yields the equivalence of (i) and (ii).
\end{proof}

\subsection{The Feichtinger conjecture} We are now ready to formulate the quantitative version of the Feichtinger conjecture which was shown by Casazza, Marcus, Speegle, and the author \cite{BCMS}.

\begin{theorem*}[$FEI$] \label{main2} 
Let $I$ be at most countable set and let $\mathcal H$ be  a separable Hilbert space. Suppose $\{u_i\}_{i\in I}$ is a Bessel sequence in $\mathcal H$ with bound $1$ that consists of vectors of norms $\|u_i\|^2 \geq \ve$, where $\ve>0$. 
Then there exists a universal constant $C>0$, such that $I$ can be partitioned into $r \leq C/\ve$ subsets $I_1, \ldots, I_r$, such that each subfamily $\{u_i\}_{i\in I_j}$, $j=1,\ldots, r$, is a Riesz sequence.
\end{theorem*}

In the proof of Theorem $(FEI)$ we shall use the following adaptation of the Schur-Horn theorem for Riesz sequences.

\begin{lemma}\label{major}
Let $S$ be a positive semi-defnite  $M\times M$ matrix with eigenvalues $\lambda_{1} \ge \ldots \ge \lambda_M \ge 0$. Let $d_1\ge \ldots \ge d_M \ge 0$ be such that 
\begin{equation}\label{major1}
\sum_{i=1}^{M}d_{i} = \sum_{i=1}^{M} \lambda_{i}\quad\text{and}\quad\sum_{i=1}^{k}d_{i}\leq \sum_{i=1}^{k}\lambda_{i}\quad\text{for all } 1\leq k\leq M.
\end{equation}
Then there exists a collection of vectors $\{v_{i}\}_{i=1}^{M}$ in $\C^M$ such its frame operator is $S$ and $\|v_{i}\|^{2} = d_{i}$ for all $i=1,\ldots,M$.
\end{lemma}

Lemma \ref{major} has a converse, which states that the norms of $\{v_{i}\}_{i=1}^{M}$ and eigenvalues of its frame operator must satisfy \eqref{major1}. Since we will not need this, we simply omit the proof of the converse result.

\begin{proof} By the Schur-Horn theorem, there exists a hermitian matrix $\tilde S$ with eigenvalues $\lambda_{1} \ge \ldots \ge \lambda_M$ and diagonal $d_1\ge \ldots \ge d_M$. Since $S$ and $\tilde S$ are unitarily equivalent, there exists a unitary $M\times M$ matrix $U$ such that $\tilde S= U^* S U$. Define vectors $v_i=S^{1/2} U e_i$, where $e_i$, $i\in [M]$, are standard basis vectors in $\C^M$. Then, 
\[
||v_i||^2= \lan S^{1/2} U e_i ,  S^{1/2} U e_i \ran = \lan S U e_i, U e_i \ran = \lan \tilde S e_i, e_i \ran = d_i.
\]
Moreover, the frame operator of $\{v_{i}\}_{i=1}^{M}$ satisfies for $v\in \C^M$,
\begin{multline*}
\bigg( \sum_{i=1}^M v_i \otimes v_i \bigg)(v)=
\sum_{i=1}^M (S^{1/2} Ue_i) \otimes (S^{1/2} U e_i) (v) = \sum_{i=1}^M \lan v, S^{1/2} U e_i \ran S^{1/2} U e_i 
\\
=  S^{1/2} \bigg(\sum_{i=1}^M \lan  S^{1/2} v,  U e_i \ran U e_i \bigg) = S^{1/2} S^{1/2} v =Sv.
\end{multline*}
The penultimate step follows from the fact that $\{U e_i\}_{i=1}^M$ is an o.n. basis.
\end{proof}

We start from the special case of Theorem $(FEI)$ and then show its increasingly general versions.

\begin{lemma}\label{ves}
Theorem $(FEI)$ holds under the additional assumption that $I$ is finite and $\ve= 0.92$. In this case, a Bessel sequence $\{u_i\}_{i\in I}$ with bound $1$ and $\|u_i\|^2 \geq \ve$ can be partitioned into two Riesz sequences with lower bound $0.02$.
\end{lemma}

In light of Remark \ref{rem}, the value of $0.92$ can be replaced by any number $>3/4$, but we are not after best constants here. 

\begin{proof}
Assume momentarily that  $\{u_i\}_{i\in I}$ is a Parseval frame in a finite dimensional Hilbert space $\mathcal H$.
By Lemma \ref{nai} we can imbed $\mathcal H$ into $\ell^2(I)$ such that $u_i=Pe_i$, $i\in I$, where $P$ is an orthogonal projection of $\ell^2(I)$ onto $\mathcal H$. Then, vectors $v_i = (\mathbf I - P)e_i$, $i\in I$, form a Parseval frame in $\ell^2(I) \ominus \mathcal H$. Since $||u_i||^2 \ge \ve$, we have $||v_i||^2 \le 1-\ve < \delta: = 0.08$. By Theorem $(KS_2^\infty)$ we can find a subset $J \subset I$, such that both $\{v_i\}_{i\in J}$ and $\{v_i \}_{i\in I \setminus J}$ are Bessel sequences with bound $\frac12 + \sqrt{2\delta} + \delta = 0.98$. Thus, by Lemma \ref{ce}, both $\{u_i\}_{i\in J}$ and $\{u_i \}_{i\in I \setminus J}$ are Riesz sequences with lower bound $1-0.98=0.02$. 

Assume now that $\{u_i\}_{i\in I}$ is a Bessel sequence with bound $1$ and $\|u_i\|^2 \geq \ve$. Since $I$ is finite, we can assume that $I=[n]$ and $\mathcal H=\C^d$. By increasing the dimension of the ambient space, we claim that a Bessel sequence $\{u_i\}_{i\in [n]}$ can be extended to a Parseval frame by adjoining some collection of vectors $\{u_{n+i}\}_{i=1}^{d+N}$ in $\C^{d+N}$ satisfying $||u_i|| \ge \ve$, where $N$ is sufficiently large.

Indeed, suppose that the frame operator of $\{u_i\}_{i=1}^n$, which acts on $\C^d$,
has eigenvalues $1\ge \lambda_1 \ge \lambda_2 \ge \ldots \ge \lambda_d \ge 0$. For a fixed $N$, consider an operator  on $\C^{d+N}$,
\[
\tilde S = \mathbf I_{d+N} - S \oplus \mathbf 0_N ,\qquad\text{where $\mathbf 0_N$ is the zero operator on $\C^N$.}
\]
Then, $\tilde S$ has the following eigenvalues listed in decreasing order 
\begin{equation}\label{PM2}
\underbrace{1,\ldots,1}_{N},1-\lambda_d, \ldots, 1-\lambda_1.
\end{equation}
Thus, we need to show the existence of vectors $\{u_{n+i}\}_{i=1}^{d+N}$ in $\C^{d+N}$ such that:
\begin{enumerate}[(i)]
\item its frame operator is $\tilde S$, and
\item
$\|u_{n+i} \|^2=C$ for all $i=1,\ldots, d+N$ for some constant $C \in [\ve,1]$.
\end{enumerate} 
By Lemma \ref{major}, this is possible provided eigenvalue sequence \eqref{PM2} majorizes, in the sense of \eqref{major1}, the sequence
\begin{equation}\label{PM3}
\underbrace{C,\ldots,C}_{d+N}.
\end{equation}
However, the majorization \eqref{PM2} is automatic for the constant sequence \eqref{PM3} provided that
\begin{equation*}
(d+N)C=N+\sum_{i=1}^d (1-\lambda_i).
\end{equation*}
Thus, by choosing sufficiently large $N$, we have $C\ge \ve$, which shows the claim.
Now, we apply the previous argument for a Parseval frame $\{u_i\}_{i=1}^{n+d+N}$. Hence, we can find a partition into two Riesz sequences. Restricting this partition to the original sequence $\{u_i\}_{i=1}^{n}$ yields the same conclusion.
\end{proof}

\begin{lemma}\label{ves2}
Theorem $(FEI)$ holds under the assumption that $I$ is finite and $||u_i||^2=\ve>0$ for all $i$. In this case, a Bessel sequence $\{u_i\}_{i\in I}$ with bound $1$ can be partitioned into two Riesz sequences with bounds $\ve/50$ and $\ve/0.92$.
\end{lemma}

\begin{proof}
By scaling Lemma \ref{ves} yields the following result: any finite tight frame $\{w_i\}$ with constant $B$ and with norms $||w_i||^2 \ge 0.92 B$ can be partitioned into two Riesz sequences with bounds $B/50$ and $B$.

Now, suppose that $\{u_i\}_{i\in I}$ is a Bessel sequence with bound $1$ and $\|u_i\|^2 = \ve$.
By Theorem $(KS_r)$ for each $r$ we can find a partition $\{\tilde I_j\}_{j=1}^{r}$ of $I$ such that each $\{u_i\}_{i\in  \tilde I_j}$  is a Bessel sequence with bound 
\[
B=\bigg(\frac{1}{\sqrt{r}} + \sqrt{\ve}\bigg)^2.
\]
Now we choose large enough $r$ such that
\begin{equation}\label{fei2}
\|u_i\|^2 = \ve \ge 0.92 B= 0.92 \bigg(\frac{1}{\sqrt{r}} + \sqrt{\ve}\bigg)^2.
\end{equation}
A simple calculation shows that the above inequality simplifies to
\[
\frac{2}{\sqrt{r \ve}} + \frac{1}{r \ve} \le \frac{0.08}{0.92}.
\]
Hence, it suffices to choose 
\[
r \ge \frac{9}{\ve} \bigg(\frac{0.92}{0.08} \bigg)^2.
\]
By Lemma \ref{ves}, each $\{u_i\}_{i\in  \tilde I_j}$ can be partitioned into two Riesz sequences with lower bound $ B/50 \ge \ve/50$ and upper bound $B\le \ve/0.92$. 
This gives the required partition of size $2r$ and completes the proof of Lemma \ref{ves2}.
\end{proof}

Theorem $(FEI)$ is now a consequence of Lemmas \ref{pinball} and \ref{ves2}.
\begin{proof}[Proof of Theorem $(FEI)$]
Suppose $\{u_i\}_{i\in I}$ is an infinite Bessel sequence in a Hilbert space $\mathcal H$ satisfying $||u_i||^2 \ge \ve$.
Without loss of generality we can assume that $||u_i||^2=\ve$ for all $i \in I$. Indeed, $\{\sqrt{\ve} \frac{u_i}{||u_i||}\}_{i\in I}$ is also Bessel sequence with bound $1$. Applying $(FEI)$ for this sequence yields the same conclusion for the original Bessel sequence $\{u_i\}_{i\in I}$.

 Since $I$ is countable, we may assume $I=\N$. For any $n\in \N$, we apply Lemma \ref{ves2} to the initial sequence $\{u_i\}_{i\in [n]}$. Hence, we find a partition $\{I_1^n, \ldots, I_r^n\}$ of $\{u_i\}_{i\in [n]}$ into Riesz sequences with uniform lower and upper bounds of $\ve/50$ and $\ve/0.92$, resp. To show the existence of a global partition of $\{I_1,\ldots,I_r\}$ of $\{u_i\}_{i\in \N}$ into Riesz sequences, it suffices to apply Lemma \ref{pinball}. This done in the same way as in the proof of Lemma \ref{ksinf}.
\end{proof}

\subsection{Casazza-Tremain conjecture} A stronger variant of the Feichtinger conjecture, called $R_\ve$ conjecture, was studied by  Casazza and Tremain \cite{CT1}. This result states that Bessel sequences consisting of unit norm vectors can be partitioned into almost orthogonal sequences. 

\begin{theorem*}[$R_\ve$]
Suppose that $\{u_i\}_{i\in I}$ is a unit norm Bessel sequence with bound $B$ in a separable Hilbert space $\mathcal H$. Then for any $\ve>0$ there exists a partition $\{I_1,\ldots, I_r\}$ of $I$ of size $r=O(B/\ve^4)$, such that each $\{u_i\}_{i\in  I_j}$, $j=1,\ldots, r$, is a Riesz sequence with bounds  $1-\ve$ and $1+\ve$.
\end{theorem*}

In the proof of Theorem ($R_\ve$) we will use the following lemma. The case when $J=I$ is a well-known fact, see \cite[Section 3.6]{Ch}. For the sake of completeness we will give the proof of Lemma \ref{subriesz}.

\begin{lemma}\label{subriesz} 
Suppose $\{u_i\}_{i\in I}$ is a Riesz basis in a Hilbert space $\mathcal H$. Let $\{u^*_i\}_{i\in I}$ be its unique biorthogonal (dual) Riesz basis, i.e.,
\[
\lan u_i, u_j^* \ran =\delta_{i,j} \qquad\text{for all }i,j\in I.
\]
Let $J\subset I$ be any subset. Then,  $\{u_i\}_{i\in J}$ is a Riesz sequence with bounds $A$ and $B$ $\iff$ $\{u^*_i\}_{i\in J}$  is a Riesz sequence with bounds $1/B$ and $1/A$.
\end{lemma}

\begin{proof} Suppose that $\{u_i\}_{i\in J}$ has upper Riesz bound $B$. This is equivalent to the Bessel condition
\begin{equation}\label{sr2}
\sum_{i\in J} |\lan u, u_i \ran |^2 \le B ||u||^2
\qquad\text{for all }u\in \mathcal H.
\end{equation}
For any sequence $\{a_i\}_{i\in J} \in \ell^2$, there exists a unique $u\in \mathcal H$ such that
\[
\lan u, u_i \ran = \begin{cases} a_i & i\in J,\\
0 & \text{otherwise}.
\end{cases}
\]
Since $u= \sum_{i\in J} a_i u^*_i$, by \eqref{sr2} we have
\[
\bigg\| \sum_{i\in J} a_i u^*_i \bigg\|^2 = ||u||^2 \ge \frac{1}{B} \sum_{i\in J} |\lan u, u_i \ran |^2 = \frac{1}{B} \sum_{i\in J} |a_i|^2.
\]
Conversely, if $\{u^*_i\}_{i\in J}$ has lower Riesz bound $1/B$, then \eqref{sr2} holds and $\{u_i\}_{i\in J}$ has upper Riesz bound $B$. 
By symmetry, $\{u^*_i\}_{i\in J}$ has upper Riesz bound $1/A$ if and only if $\{u_i\}_{i\in J}$ has lower Riesz bound $A$, which completes the proof of the lemma.
\end{proof}

\begin{lemma} 
$(FEI) \implies (R_\ve)$.
\end{lemma}

\begin{proof}
In the first step we apply  a scaled version of Theorem $(FEI)$  to find a partition of $\{u_i\}_{i\in I}$ of size $O(B)$ into Riesz sequences with uniform lower and upper bounds. By Lemma \ref{ves2}, these bounds are $1/50$ and $1/0.92$, resp. 

Suppose that $\{u_i\}_{i \in I'}$, $I' \subset I$ is one of these unit-norm Riesz sequences. In the next step we need to tighten these bounds as follows. 
Let $\{u^*_i\}_{i\in I'} $ be the unique biorthogonal (dual) Riesz basis to $\{u_i\}_{i\in I'}$ in its closed linear span $\mathcal H'=\overline{\operatorname{span}}\{u_i: i \in I'\}$. By Lemma \ref{subriesz} the upper Riesz bound of $\{u^*_i\}_{i\in I'}$ is $50$. Applying Theorem $(KS_r)$ to both $\{u_i\}_{i\in I'}$ and $\{u^*_i\}_{i\in I'}$, we can find partitions into Riesz sequences, which reduce upper bounds to $1+\ve$. A calculation shows that this requires partitions each of size $O(1/\ve^2)$. Taking common refinement of all of these partitions produces a partition of $\{u_i\}_{i\in I}$ of size $O(B/\ve^4)$. Let $\{u_i\}_{i\in J}$ be any element of of this partition. Then, both $\{u_i\}_{i\in J}$ and $\{u^*_i\}_{i\in J}$ are Riesz sequences with upper bounds $1+\ve$. Lemma \ref{subriesz} implies that $\{u_i\}_{i\in J}$ has lower bound $1/(1+\ve) \ge 1 - \ve$.  
\end{proof}

\subsection{Bourgain-Tzafriri conjecture}

Theorem $(R_\ve)$ yields automatically the Bourgain--Tzafriri restricted invertibility conjecture.

\begin{theorem*}[$BT$]
Let $\{e_i\}_{i\in I}$ be an orthonormal basis of a separable Hilbert space $\mathcal H$. Let $T: \mathcal H \to \mathcal H$ be a bounded linear operator  with norm $\|T\|^2 \le B$ and $\|Te_i\|=1$ for all $i\in I$, where $B>1$.
Then, for any $\ve>0$, there exists a partition $\{I_1,\ldots, I_r\}$ of $I$ of size $r=O(B/\ve^4)$, such that $T$ is $(1+\ve)$-isometry when restricted to each orthogonal subspace
\[
\mathcal H_j = \overline{\operatorname{span}} \{ e_i : i \in I_j\}.
\]
That is, for all  $j =1,\ldots,r$,
\begin{equation}\label{bt}
(1-\ve) ||f||^2 \le ||Tf||^2  \le (1+\ve)||f||^2 \qquad\text{for all }f \in \mathcal H_j.
\end{equation}
\end{theorem*}

\begin{lemma} $(R_\ve) \implies (BT)$.
\end{lemma}

\begin{proof}
Define vectors $u_i = Te_i$, $i\in I$. By our hypothesis $\{u_i\}_{i\in I}$ is a unit norm Bessel sequence with bound $B$. By Theorem $(R_\ve)$, there exists a partition  $\{I_1,\ldots, I_r\}$ of $I$ of size $r=O(B/\ve^4)$ such that each collection $\{u_i\}_{i\in I_j}$, $j\in [r]$, is a Riesz sequence with bounds $1-\ve$ and $1+\ve$. Translating this back for the property of $T$ yields \eqref{bt}.
\end{proof}

A classical application of the results studied in this section involves Fourier frames. If $E \subset [0,1]$ has positive Lebesgue measure, then the collection of functions $\phi_n(t)=e^{2\pi i nt} \chi_E(t)$, $n\in \Z$, is a Parseval frame for $L^2(E)$, often called a Fourier frame. Since this is an equal norm frame, i.e., $||\phi_n||^2=|E|$ for all $n\in \Z$, Theorem $(R_\ve)$ yields the following corollary. 

\begin{corollary}\label{ff} There exists a universal constant $c>0$ such that for any $\ve>0$ and any subset $E \subset [0,1]$ with positive measure, the corresponding Fourier frame $\{\phi_n\}_{n\in\Z}$ can be decomposed as the union of $r \le c\ve^{-4}|E|^{-1}$ Riesz sequences with bounds $1\pm \ve$.
\end{corollary}

\subsection{Notes} The proof of $(FEI)$ and $(R_\ve)$ follows the approach in \cite{BCMS} with some minor simplifications. One can show $(FEI)$ with less effort by deducing it from Theorem $(PB)$ as in \cite[Proposition 3.1]{CCLV}, but with worse bounds on the partition size $r$. The bound on $r$ in Theorem $(FEI)$ is asymptotically optimal as $\ve \to 0$. This can be seen by considering a union of $\lfloor 1/\ve \rfloor$ o.n. bases scaled by the factor $\sqrt{\ve}$. A more general version of Lemma \ref{major} for frames can be found in \cite{AMRS, BJ}.

Lawton \cite{La} and Paulsen \cite{Pa} have shown that the partition subsets of $\mathbb Z$ in Corollary \ref{ff} can be chosen to be syndetic sets, i.e., subsets of $\Z$ with bounded gaps.
The study of the Feichtinger conjecture for Fourier frames is connected with the problem of paving for Laurent operators. A Laurent operator $L_\varphi: L^2[0,1] \to L^2[0,1]$ is given by $L_\varphi f = \varphi f$ for $f\in L^2[0,1]$, where the symbol $\varphi \in L^\infty[0,1]$. The problem of paving for Laurent operators was studied by Halpern, Kaftal, and Weiss \cite{HKW}. It was continued by Bourgain and Tzafriri \cite{BT3}, who have shown that every Fourier frame has a Riesz sequence indexed by a subset $\Lambda \subset \Z$ of positive upper density at least $c|E|$. This is the consequence of their celebrated restricted invertibility theorem \cite{BT1, BT2}, which also holds for $\ell^p$ spaces. While the Bourgain-Tzafriri restricted invertibility conjecture, Theorem $(BT)$, holds for $\ell^p$ when $p=2$, it is an open problem for $p\ne 2$.

Akemann and Weaver \cite{AW} have shown an interesting generalization of $(KS_r)$ in the form of Lyapunov's theorem.
\begin{theorem*}[$AW$]
Suppose $\{u_i\}_{i\in I}$ is a Bessel family with bound $1$ in a separable Hilbert space $\mathcal{H}$, which consists of vectors of norms $\|u_i\|^2\leq \delta$, where $\delta>0$. Suppose that $0\le \tau_i \le 1$ for all $i\in I$. Then, there exists a subset of indices $I_0 \subset I$ such that
\begin{equation}\label{aw0}
\bigg\| \sum_{i\in I_0} u_i \otimes u_i - \sum_{i\in I} \tau_i u_i \otimes u_i \bigg\| \le C \delta^{1/8},
\end{equation}
where $C>0$ is a universal constant.
\end{theorem*}

The proof of Theorem $(AW)$ relies solely on $(KS_r)$, and hence we could have added another implication to our diagram $(KS_r) \implies (AW)$. However, we will stop here and instead invite the reader to explore other interesting consequences of the breakthrough solution of the Kadison-Singer problem.

\bibliographystyle{amsplain}

\begin{thebibliography}{99}

\bibitem{AA}
C. Akemann, J. Anderson, 
{\it Lyapunov theorems for operator algebras},
Mem. Amer. Math. Soc. {\bf 94} (1991), no. 458, iv+88 pp. 

\bibitem{AW}
C. Akemann, N. Weaver, 
{\it A Lyapunov-type theorem from Kadison-Singer},
Bull. Lond. Math. Soc. {\bf 46} (2014), no. 3, 517--524. 


\bibitem{An1}
J. Anderson, 
{\it Extreme points in sets of positive linear maps on $\mathcal B(\mathcal H)$},
J. Funct. Anal. {\bf 31} (1979), no. 2, 195--217.

\bibitem{An2}
J. Anderson,
{\it Extensions, restrictions, and representations of states on $C^*$-algebras},
Trans. Amer. Math. Soc. {\bf 249} (1979), no. 2, 303--329. 

\bibitem{An3}
J. Anderson,
{\it A conjecture concerning the pure states of $\mathcal B(\mathcal H)$ and a related theorem}, Topics in modern operator theory (Timi\c soara/Herculane, 1980), 27--43,
Operator Theory: Adv. Appl., 2, Birkh\"auser, Basel-Boston, Mass., 1981. 

\bibitem{AMRS}
J. Antezana, P. Massey, M. Ruiz, D. Stojanoff, 
{\it 
The Schur-Horn theorem for operators and frames with prescribed norms and frame operator}, Illinois J. Math. {\bf 51}
(2007), 537--560.

\bibitem{BHKW}
K. Berman, H. Halpern, V. Kaftal and G. Weiss, 
{\it Matrix norm inequalities and the relative Dixmier property}, Integral Equations Operator Theory {\bf 11} (1988), 28--48.

\bibitem{BHKW2}
K. Berman, H. Halpern, V. Kaftal and G. Weiss, {\it Some $C_4$ and $C_6$ norm inequalities related to the paving problem},
 Operator Theory: Operator Algebras and Applications, Part 2 (Durham, NH, 1988), 29--41, Proc. Sympos. Pure Math., vol. 51, Amer. Math. Soc., Providence, RI 1990. 



\bibitem{BB}
J. Borcea, P. Br\"and\'en, 
{\it Multivariate P\'olya-Schur classification problems in the Weyl algebra},
Proc. Lond. Math. Soc. {\bf 101} (2010), 73--104. 

\bibitem{BT1}
J. Bourgain, L. Tzafriri, 
{\it Invertibility of ``large'' submatrices with applications to the geometry of Banach spaces and harmonic analysis},
Israel J. Math. {\bf 57} (1987), no. 2, 137--224. 

\bibitem{BT2}
J. Bourgain, L. Tzafriri, 
{\it Restricted invertibility of matrices and applications}, Analysis at Urbana, Vol. II (Urbana, IL, 1986--1987), 61--107,
London Math. Soc. Lecture Note Ser., 138, Cambridge Univ. Press, Cambridge, 1989. 

\bibitem{BT3}
J. Bourgain, L. Tzafriri, 
{\it On a problem of Kadison and Singer},
J. Reine Angew. Math. {\bf 420} (1991), 1--43. 


\bibitem{BCMS}
M. Bownik, P. Casazza, A. Marcus, D. Speegle, {\it Improved bounds in Weaver and Feichtinger conjectures}, J. Reine Angew. Math. (to appear).

\bibitem{BJ}
M. Bownik, J. Jasper,
{\it Existence of frames with prescribed norms and frame operator}, Excursions in harmonic analysis. Volume 4, 103--117, Appl. Numer. Harmon. Anal., Birkh\"auser/Springer, New York, 2015.

\bibitem{BS}
M. Bownik, D. Speegle, 
{\it The Feichtinger conjecture for wavelet frames, Gabor frames and frames of translates},
Canad. J. Math. {\bf 58} (2006), no. 6, 1121--1143.

\bibitem{CCLV}
P. Casazza, O. Christensen, A. Lindner, R. Vershynin, 
{\it Frames and the Feichtinger conjecture},
Proc. Amer. Math. Soc. {\bf 133} (2005), no. 4, 1025--1033. 

\bibitem{CE}
P. Casazza, D. Edidin, 
{\it Equivalents of the Kadison--Singer problem},
Function spaces, 123--142,
Contemp. Math., 435, Amer. Math. Soc., Providence, RI, 2007. 

\bibitem{CEKP}
P. Casazza, D. Edidin, D. Kalra, V. Paulsen, 
{\it Projections and the Kadison--Singer problem},
Oper. Matrices {\bf 1} (2007), no. 3, 391--408.

\bibitem{CFTW}
P. Casazza, M. Fickus, J. Tremain, E. Weber,
{\it The Kadison--Singer problem in mathematics and engineering: a detailed account}, Operator theory, operator algebras, and applications, 299--355, Contemp. Math., 414, Amer. Math. Soc., Providence, RI, 2006.

\bibitem{CFMT}
P. Casazza, M. Fickus, D. Mixon and J. Tremain, 
{\it The Bourgain-Tzafriri Conjecture and concrete construction of non-pavable projections},
Oper. Matrices {\bf 5} (2011), no. 2, 353--363.

\bibitem{CT}
P. Casazza, J. Tremain, 
{\it The Kadison--Singer problem in mathematics and engineering},
Proc. Natl. Acad. Sci. USA {\bf 103} (2006), no. 7, 2032--2039.

\bibitem{CT1}
P. Casazza and J. Tremain, {\it Revisiting the Bourgain-Tzafriri Restricted Invertibility Theorem}, Oper.  Matrices {\bf 3} (2009), no. 1, 97--110.

\bibitem{CT2}
P. Casazza and J. Tremain,{\it Consequences of the Marcus/Spielman/Srivastava solution of the Kadison-Singer problem}. New trends in applied harmonic analysis, 191--213, Appl. Numer. Harmon. Anal., Birkh\"auser/Springer, Cham, 2016.

\bibitem{Ch}
O. Christensen,
{\it An introduction to frames and Riesz bases}, Applied and Numerical Harmonic Analysis. Birkh\"auser Boston, Inc., Boston, MA, 2003.

\bibitem{FG}
K. Fritzsche, H. Grauert,
{\it From holomorphic functions to complex manifolds},
Graduate Texts in Mathematics, 213. Springer-Verlag, New York, 2002. xvi+392 pp.

\bibitem{Gr}
K. Gr\"ochenig, 
{\it Localized frames are finite unions of Riesz sequences},
Adv. Comput. Math. {\bf 18} (2003), no. 2--4, 149--157. 

\bibitem{HKW} 
H. Halpern, V. Kaftal and G. Weiss, {\it Matrix pavings and Laurent operators}, J. Operator Theory {\bf 16} (1986), no. 2, 355--374.

\bibitem{HKW2} 
H. Halpern, V. Kaftal and G. Weiss, {\it The relative Dixmier property in discrete crossed products},
J. Funct. Anal. {\bf 69} (1986), no. 1, 121--140. 

\bibitem{HKW3} 
H. Halpern, V. Kaftal and G. Weiss, {\it Matrix pavings in $\mathcal B(\mathcal H)$},  Operators in indefinite metric spaces, scattering theory and other topics (Bucharest, 1985), 201--214, Oper. Theory Adv. Appl., 24, Birkh\"auser, Basel, 1987.


\bibitem{Ha}
N. Harvey, {\it An introduction to the Kadison-Singer Problem and the Paving Conjecture}, preprint available at {\tt https://www.cs.ubc.ca/~nickhar/Publications/KS/KS.pdf}.


\bibitem{HV}
J. W. Helton, V. Vinnikov, 
{\it Linear matrix inequality representation of sets},
Comm. Pure Appl. Math. {\bf 60} (2007), no. 5, 654--674.

\bibitem{La}
W. Lawton, 
{\it Minimal sequences and the Kadison--Singer problem},
Bull. Malays. Math. Sci. Soc. {\bf 33} (2010), 169--176. 

\bibitem{KS}
R. Kadison, I. Singer,
{\it Extensions of pure states},
Amer. J. Math. {\bf 81} (1959), 383--400. 

\bibitem{MSS0}
A. W. Marcus, D. A. Spielman, N. Srivastava, 
{\it Interlacing families I: bipartite Ramanujan graphs of all degrees}, Ann. of Math. {\bf 182} (2015), no. 1, 307--325.

\bibitem{MSS}
A. W. Marcus, D. A. Spielman, N. Srivastava, {\it Interlacing Families II: mixed characteristic polynomials and the Kadison--Singer problem}, Ann. of Math. {\bf 182} (2015), no. 1, 327--350.

\bibitem{Mat}
\'E. Matheron, 
{\it Le probl\`eme de Kadison-Singer},
Ann. Math. Blaise Pascal {\bf 22 }(2015), no. S2, 151--265. 


\bibitem{Pa}
V. Paulsen,
{\it Syndetic sets, paving and the Feichtinger conjecture},
Proc. Amer. Math. Soc. {\bf 139} (2011), 1115--1120. 

\bibitem{Rav}
M.  Ravichandran, {\it Mixed Determinants and the Kadison-Singer problem}, preprint (2016) available at {\tt arXiv:1609.04195}.

\bibitem{St}
M. Stevens, {\it The Kadison-Singer property. With a foreword by Klaas Landsman}, SpringerBriefs in Mathematical Physics, 14. Springer, Cham, 2016.


\bibitem{Sch}
G. Schechtman, 
{\it Three observations regarding Schatten $p$ classes},
J. Operator Theory {\bf 75} (2016), 139--149.

\bibitem{Tan}
B. Tanbay, {\it A letter on the Kadison-Singer problem}, Rev. Roumaine Math. Pures Appl. {\bf 59} (2014), no. 2, 293--302.

\bibitem{Tao}
T. Tao, {\it Real stable polynomials and the Kadison-Singer problem}, blog entry available at {\tt https://terrytao.wordpress.com/tag/kadison-singer-problem/}

\bibitem{Ti}
D. Timotin,
{\it The solution to the Kadison--Singer Problem: yet another presentation}, preprint {\tt arXiv:1501.00464}.

\bibitem{Va}
A. Valette, 
{\it Le probl\`eme de Kadison-Singer (d'apr\`es A. Marcus, D. Spielman et N. Srivastava)},
Ast\'erisque No. 367--368 (2015), Exp. No. 1088, x, 451--476.

\bibitem{We}
N. Weaver, 
{\it The Kadison--Singer problem in discrepancy theory},
Discrete Math. {\bf 278} (2004), no. 1--3, 227--239. 









\end{thebibliography}

\end{document}